\documentclass[12pt, reqno]{article}

\usepackage{amsmath, amsthm, amssymb}
\usepackage{blindtext}
\usepackage{bbm}
\usepackage{enumerate}

\usepackage{float}
\usepackage{fullpage}
\usepackage{enumerate,tcolorbox}
\usepackage{mathtools}

 \newtheorem{lemma}{Lemma}
 \newtheorem{theorem}{Theorem}

 \newtheorem{definition}{Definition}
 \newtheorem{remark}{Remark}
 \newtheorem{example}{Example}

\def\semicolon{;}
\def\applytolist#1{
    \expandafter\def\csname multi#1\endcsname##1{
        \def\multiack{##1}\ifx\multiack\semicolon
            \def\next{\relax}
        \else
            \csname #1\endcsname{##1}
            \def\next{\csname multi#1\endcsname}
        \fi
        \next}
    \csname multi#1\endcsname}

\def\calc#1{\expandafter\def\csname c#1\endcsname{{\mathcal #1}}}
\applytolist{calc}QWERTYUIOPLKJHGFDSAZXCVBNM;
\def\bbc#1{\expandafter\def\csname bb#1\endcsname{{\mathbb #1}}}
\applytolist{bbc}QWERTYUIOPLKJHGFDSAZXCVBNM;
\def\bfc#1{\expandafter\def\csname bf#1\endcsname{{\mathbf #1}}}
\applytolist{bfc}QWERTYUIOPLKJHGFDSAZXCVBNM;
\def\sfc#1{\expandafter\def\csname s#1\endcsname{{\sf #1}}}
\applytolist{sfc}QWERTYUIOPLKJHGFDSAZXCVBNM;
\def\fc#1{\expandafter\def\csname f#1\endcsname{{\mathfrak #1}}}
\applytolist{fc}QWERTYUIOPLKJHGFDSAZXCVBNM;

\usepackage{xcolor}

\usepackage{float}
\usepackage{tikz}
\usepackage[section]{placeins}

\title{Cutoff for Contingency Table and Torus Random Walks with Low Incremental Correlations}
\author{
Zihao Fang \thanks{Department of Mathematics, The Ohio State University, USA, email: fang.735@osu.edu} 
\qquad 
Andrew Heeszel \thanks{Department of Statistics, The Ohio State University, USA, email: heeszel.1@osu.edu}}
\date{}

\begin{document}
\maketitle

  \begin{abstract}
  We use the correlation matrix of the generating distribution to determine the mixing time for random walks on the torus $(\bbZ/q\bbZ)^n$.
  We present our method in the context of the Diaconis-Gangolli random walk on both the $1 \times n$ and $m \times n$ contingency tables over $\bbZ/q\bbZ$.
In the $1 \times n$ case, we prove that the random walk exhibits cutoff at time $\dfrac{n q^2 \log(n)}{8 \pi^2}$ when $q \gg n$; 
in the $m \times n$ case, where $m, n$ are of the same order, we establish cutoff for the random walk at time $\dfrac{mn q^2 \log(mn)}{16 \pi^2}$ when $q \gg n^2$. 
Our method reveals that a general class of random walks on the torus $(\bbZ/q\bbZ)^n$ has cutoff. 
If each coordinate of the lifted random walk onto $\bbZ^n$ has variance $\sigma^2/n$ in each jump, and the between-coordinate correlations are sufficiently low, then cutoff occurs at time $\dfrac{nq^2 \log(n)}{4\pi^2 \sigma^2}$.
  \end{abstract}





\tableofcontents 

\section{Introduction}

\subsection{Motivation and Basic Definitions}

We are interested in showing that a class of torus random walks has cutoff primarily based on the correlation matrix of the incremental distribution of the walk.
We consider a continuous-time random walk $\{ X_t \}_{t \geq 0}$ on the torus $(\bbZ/q\bbZ)^n$ when $q$ is large with respect to $n$, and the correlations between coordinates in the incremental distribution are sufficiently small.  
We provide sufficient conditions for cutoff based on the correlation matrix of $X_t$ when lifted to $\mathbb{Z}^n$ and the decay of its characteristic function.
This setup covers various random walks currently studied in the literature, including the Diaconis-Gangolli random walk on contingency tables over $\bbZ_q := \bbZ/q\bbZ$. 
We will use the Diaconis-Gangolli random walk as a primary model to present our method.

A contingency table is a matrix of integer entries with prescribed row and column sums. 
Enumerating the contingency tables with given marginals is closely related to counting a variety of other combinatorial objects, including permutations with descent restrictions, double cosets of the symmetric groups, certain quantities about the induced representations of the symmetric groups, and some special types of Young tableaux.
The name ``contingency table" is commonly used in the context of statistics and is linked to Fisher's exact test of association and goodness of fit, along with the approximate $\chi^2$ test. 
Random sampling from the set of contingency tables with fixed margins can be used to form an approximate p-value for the test of association when the $\chi^2$ approximation is not well met.


Markov chain Monte Carlo (MCMC) methods can be used to randomly sample and approximately count contingency tables with fixed row and column sums.
One of the famous MCMC algorithms was described in \cite{DiaconisGangolli} and \cite{DiaconisSturmfels}, giving rise to the Diaconis-Gangolli random walk on contingency tables over $\bbZ_q$.
Because the sums of rows and columns are fixed, we can view the dynamics as a random walk on the torus by ignoring the last row and column.
We will first use the Diaconis-Gangolli random walk on $1 \times n$ contingency tables over $\bbZ_q$ to outline how we can use incremental correlations to establish cutoff for torus random walks (Theorem \ref{main1}).
Then we will present our method in full generality to prove cutoff for a general class of random walks on the torus $\bbZ_q^n$ satisfying assumptions on the correlations between the coordinates of a single increment (Theorem \ref{main2}).
Finally, we apply our method to the Diaconis-Gangolli random walks on the more complicated $m \times n$ contingency tables over $\bbZ_q$
and illustrate the cutoff phenomenon in this model (Theorem \ref{main3}).

In particular, the last part of this paper could be viewed as a complement to a previous work by Nestoridi and Nguyen \cite{NestoridiNguyen}.
They established cutoff for the discrete-time Diaconis-Gangolli random walk on $n \times n$ contingency tables over $\bbZ_q$, as $n\to \infty$.
When $m = n$, our Theorem \ref{main3} gives cutoff for this model in continuous time when $q$ is sufficiently large relative to $mn = n^2$, while Nestoridi and Nguyen needed $\log(q) = o \left( \dfrac{\sqrt{\log(n)}}{\log \log n} \right)$ to achieve cutoff.
Note that we need $q$ to grow faster than $n^2$ to establish cutoff, whereas Nestoridi and Nguyen required $q$ to be very small.

In addition, our general result Theorem \ref{main2} expands on examples of cutoff for non-local random walks on $\mathbb{Z}_q^n$. 
Previously, Nestoridi \cite{Nestoridi} found mixing times and regimes of cutoff for a non-local random walk on $\bbZ_2^n = (\bbZ/2\bbZ)^n$ such that $k$ coordinates are randomly selected and flipped in a single iteration. 
In contrast, our results show cutoff for random walks on $\bbZ_q^n$ with non-nearest-neighbor steps when both $n \to \infty$ and $q = q(n) \to \infty$.

We now define the dynamics of the Diaconis-Gangolli random walk, and note that we will study this model in continuous time. 
We start with the $1 \times n$ contingency tables.
\begin{definition}
    A $1\times n$ \textit{contingency table} over $\bbZ_q$ is an $n$-dimensional vector with entries from $\bbZ_q$ such that the sum of all the coordinates is fixed.
\end{definition}
We will consider the following rate-$1$ continuous-time random walk defined on $1\times n$ contingency tables over $\bbZ_q$.
Each jump of this random walk is governed by the following rule.
\begin{enumerate}[(i)]
    \item 
    Select two coordinates uniformly at random without replacement. 
    \item 
    Flip a fair coin.
    If heads, increment the first coordinate selected by $+1$ and
    increment the second coordinate selected by $-1$;
    if tails, increment the first coordinate selected by $-1$ and
    increment the second coordinate selected by $+1$.
\end{enumerate}
The analogous definition and dynamics for $m \times n$ contingency tables are the following.
\begin{definition}
    An $m\times n$ \textit{contingency table} over $\bbZ_q$ is an $m \times n$ matrix with entries from $\bbZ_q$ such that the row sums $(r_1,\dots, r_m) \in \bbZ_q^m$ and column sums $(c_1,\dots,c_n) \in \bbZ_q^n$ are each fixed.
\end{definition}
We will study the rate-1 continuous-time Diaconis-Gangolli random walk on $m \times n$ contingency tables over $\bbZ_q$, each of whose jumps is governed by the following rule:
\begin{enumerate}[(i)]
    \item Select two distinct rows $i < i'$ and two distinct columns $j < j'$ independently and uniformly at random.
    \item Flip a fair coin, and consider the  $m \times n$ matrix $T = T(i,i',j,j')$ whose $(i,j)$-th and $(i',j')$-th entries are $+1$, $(i,j')$-th and $(i',j)$-th entries are $-1$, and zero elsewhere, 
    \[
    T = \begin{pmatrix}
        1 & -1 \\
        -1 & 1
    \end{pmatrix}.
    \]
 If heads, add $T$ to the current matrix;
    if tails, add $-T$ to the current matrix.
\end{enumerate}
As mentioned above, we specifically consider the case when $m$ has the same order $n$.
More precisely, we assume that there are universal constants $\kappa_1, \kappa_2$ such that $0 < \kappa_1 \leq 1 \leq \kappa_2$ and
\[
\kappa_1 n \leq m \leq \kappa_2 n.
\]
For convenience, we will also assume that $\kappa_1 n, \kappa_2 n \in \bbZ$, and $\kappa_1 (n-1) \leq m-1 \leq \kappa_2 (n-1)$.
Since we will take $n \to \infty$, these assumptions will not affect our results.

Both of these two random walks converge to the uniform distribution on the respective state space with respect to the total variation distance.
More specifically, if the random walk starts at $x \in \Omega$, where $\Omega$ is the state space, then we can denote by $H_t(x, y)$ the probability of the walk going from $x$ to $y \in \Omega$ at time $t \geq 0$.
Let $\pi$ be the uniform measure on $\Omega$, and write
\[
d(t) := \max_{x \in \Omega}\|H_t(x, \cdot) - \pi\|_{\operatorname{TV}}
\]
where
\[
    \|\mu - \nu\|_{\operatorname{TV}} := \max_{A \subset \Omega} |\mu(A) - \nu(A)|.
\]
is the total variation distance between two probability measures $\mu, \nu$ on $\Omega$.
By the convergence theorem of Markov chains, $d(t) \to 0$ as $t \to \infty$.
Our goal is to determine the rate of convergence, which could be quantified by the $\varepsilon$-mixing time
\[
t_{\operatorname{mix}}(\varepsilon) := \inf\{t \geq 0 : d(t) \leq \varepsilon\}
\]
for $\varepsilon \in (0,1)$.
In particular, we show that both the walks on $1 \times n$ and $m \times n$ contingency tables exhibit \textit{cutoff} as $n \to \infty$.
We say that a sequence of Markov chains indexed by $n = 1,2, \ldots$ has cutoff if for all $\varepsilon \in (0,1)$,
\begin{align*}
    \lim_{n \rightarrow \infty} \dfrac{t_{\operatorname{mix}}^{(n)}(\varepsilon)}{t_{\operatorname{mix}}^{(n)}(1-\varepsilon)} = 1.
\end{align*}
In other words, for large enough $n$, there will be a sharp transition for $d(t)$ from almost $1$ to almost $0$.

\subsection{Notations}

Throughout the paper, we will use $\bbZ_q$ to denote the cyclic group $\bbZ/q\bbZ$ of $q$ elements.
We do not need $q$ to be a prime number, but we intrinsically require $q = q(n)$ to increase along with $n$.
In the case of $m \times n$ contingency table, $m = m(n)$ should also be considered as a function of $n$ such that $\kappa_1 n \leq m \leq \kappa_2 n$.

We will use the standard asymptotic notation:
\begin{itemize}
    \item 
    $f = \cO(g)$, or equivalently, $f \lesssim g$, if $\limsup f/g < \infty$, and
    \item 
    $f = o(g)$, or equivalently, $f \ll g$, if $\lim f/g = 0$.
\end{itemize}
All the limits are taken as $n \to \infty$ unless otherwise specified.
Lastly, we will write $\bbN = \{0,1,2,3,\dots\}$ for the set of nonnegative integers.

\subsection{Main Results}

Our main results are the following.
\begin{theorem}
\label{main1}
Let $\varepsilon \in (0,1)$ be arbitrary.
For the rate-$1$ continuous-time Diaconis-Gangolli random walk on $1\times n$ contingency tables over $\bbZ_q$, we have the following.
\begin{enumerate}[(a)]
    \item 
    Let
    \[
    \alpha := \inf\left\{z \geq 1 : 4\varepsilon^2 \geq \exp\left(\frac{4 + 4e^{2/e}}{\sqrt{z}}\right) - 1\right\}.
    \]
    When $t = \dfrac{nq^2 \log(\alpha n) }{8 \pi^2}\left(1 - \dfrac{1}{3q}\right)^{-1}$,
    \[
    d(t) \leq \varepsilon + o(1)
    \]
    as $n \to \infty$ with $q \gg n$.
    
    \item 
    On the other hand, as $n \to \infty$ with $q \gg \sqrt{\log(n)}$,
    \[
    t_{\operatorname{mix}}(\varepsilon) \geq \frac{nq^2}{8\pi^2}\left(\log(n) + \log\left(1 - \frac{1}{n}\right) + \log\left(\frac{\varepsilon^{-1}-1}{5}\right)\right).
    \]
\end{enumerate}    
\end{theorem}

\vspace{5mm}

Theorem \ref{main1} implies that as $n \rightarrow \infty$ with $q \gg n$, this family of continuous-time random walks on $1\times n$ contingency tables exhibits cutoff at the mixing time $\dfrac{n q^2 \log(n)}{8 \pi^2}$, with a cutoff window of order $\mathcal{O}(nq^2)$.

\begin{theorem}
\label{main3}
Let $\varepsilon \in (0,1)$ be arbitrary, and suppose there exist universal constants $\kappa_1,  \kappa_2 > 0$ so that $0 < \kappa_1 \leq 1 \leq \kappa_2$ and $\kappa_1 n \leq m \leq \kappa_2 n$.
For the rate-$1$ continuous-time Diaconis-Gangolli random walk on $m \times n$ contingency tables over $\bbZ_q$, we have the following.
\begin{enumerate}[(a)]
    \item 
    Let
    \[
\alpha := \inf\left\{z \geq 1 : 4\varepsilon^2 \geq \exp\left(\frac{\kappa_2^2(6 + 12e^{4/e} + 12e^{8/e} + 6e^{32/e})}{\kappa_1 z^{1/4}}\right) - 1\right\}.
    \]
    When $t = \dfrac{m n q^2 \log(\alpha (m-1)(n-1))}{16 \pi^2}\left(1 - \dfrac{4}{3q}\right)^{-1}$,
    \[
    d(t) \leq \varepsilon + o(1)
    \]
    as $n \to \infty$ with $q \gg mn$.
    
    \item 
    On the other hand, as $n \to \infty$ and $q \to \infty$ with $q \gg \sqrt{\log(n)}$,
    \[
    t_{\operatorname{mix}}(\varepsilon) \geq \frac{mn q^2}{16\pi^2}\left(\log\left( (m-1)(n-1) \right) + \log\left(\frac{\varepsilon^{-1}-1}{5}\right)\right).
    \]
\end{enumerate}    
\end{theorem}

Theorem \ref{main3} implies that as $n \rightarrow \infty$ with $q \gg mn$, this family of continuous-time random walks on $m\times n$ contingency tables exhibits cutoff at the mixing time $\dfrac{mn q^2 \log(mn)}{16 \pi^2}$, with a cutoff window of order $\mathcal{O}(mnq^2)$.

As mentioned in Section 1.1, we also establish cutoff for a general class of random walks on the torus $\bbZ_q^n$. 
We will now give a shortened version of the generalization, and see Section 5.4 for the full statement of the theorem. 
In our setup, we assume that $\{ X_t \}_{t \geq 0}^{(n)}$ is an irreducible rate-1 continuous-time random walk on the torus $\bbZ_q^n$ with symmetric generating distribution $\mu^{(n)}$, and let $Y^{(n)}$ be a random vector defined on $\mathbb{Z}^n$ with the law of $\mu^{(n)}$ lifted to $\mathbb{Z}^n$. 
We assume the coordinates of $Y^{(n)}$ have a marginal incremental variance $\dfrac{\sigma^2}{n}$, with $\sigma^2 = \sigma^2(n)$, and $Y^{(n)}$ has correlation matrix $\Gamma_n$. 
We will also assume that $\| Y^{(n)} \|_{1} \leq r = r(n)$, the sequence $\{ \Gamma_n \}_{n=1}^\infty$ satisfies the \textit{Correlation Condition}, and the sequence $\{ Y^{(n)}\}_{n=1}^\infty$ meets the \textit{Decay Condition} in terms of its characteristic function (see Section 5 for the formal definitions of both conditions). 
In our regime, we assume that both $n \rightarrow \infty$ and $q \rightarrow \infty$. 
Based on this setup, we can now state a shortened version of the general theorem.
\begin{theorem}
    Let $\{ X_t \}^{(n)}$ be an arbitrary family of irreducible symmetric continuous-time Markov chains on the torus $\mathbb{Z}_q^n$ with the setup above, assuming $Y^{(n)}$ has correlation matrix $\Gamma_n$. 
    Let $\varepsilon \in (0,1)$ be arbitrary. Then we have the following.
    \begin{enumerate}[(a)]
        \item 
        Suppose that $\{Y^{(n)}\}$ meets the Decay Condition and $\{ \Gamma_n \}_{n=1}^\infty$ meets the correlation condition. Then there exists $\alpha = \alpha(\varepsilon)$ so that when
        \begin{align*}
            t = \dfrac{n q^2 \log(\alpha n )}{4 \pi^2 \sigma^2} \left( 1 - \dfrac{r^2}{12 q} \right)^{-1}
        \end{align*}
        Then,
        \begin{align*}
            d^{(n)}(t) \leq \varepsilon + o(1)
        \end{align*}
        as $n \rightarrow \infty$, $q \rightarrow \infty$ with $r^2 \ll q$.
        \item 
        On the other hand as $n \rightarrow \infty$, $q \rightarrow \infty$ with $\sup_{k \neq \ell} | \Gamma_{n, k \ell}| \ll \dfrac{1}{\log(n)}$, and $\dfrac{r^2}{q^2} \ll \dfrac{1}{\log(n)}$,
        \begin{align*}
            t_{\operatorname{mix}}^{(n)}(\varepsilon) \geq \dfrac{n q^2}{4 \pi^2 \sigma^2} \left( \log(n) + \log \left( \frac{\varepsilon^{-1} -1}{5} \right) \right)
        \end{align*}
    \end{enumerate}
\end{theorem}
We see that if the family of Markov chains $\{X^{(n)}_t \}_{n=1}^\infty$ meets the assumptions of Theorem $3$, then $\{X^{(n)}_t \}_{n=1}^\infty$ has cutoff at the mixing time $\dfrac{n q^2 \log(n)}{4 \pi^2 \sigma^2}$, with a cutoff window of order $\dfrac{n q^2}{\sigma^2}$. 
Here our primary regularity conditions rely on the correlations of the walk being sufficiently low, and the generating distribution of the walk having a fast enough rate of decay outside a neighborhood of $0$. 
We will postpone a more precise description of the general result to the end of Section 5.

\subsection{Background}






This sampling method for contingency tables via the described Markov chain had been applied in practice, but it was first formally discussed by Diaconis and Gangolli \cite{DiaconisGangolli} and Diaconis and Saloff-Coste (page 373 in \cite{DiaconisSturmfels}), where the entries of the contingency tables were nonnegative integers. 
For $m \times n$ contingency tables over $\bbN$ with $m, n$ fixed, Diaconis and Saloff-Coste showed that the mixing time of the Markov chain is of order $N^2$, where $N := \sum c_i = \sum r_i$ is the sum of all the entries of the table.
Hernek \cite{Hernek} specifically investigated the model on $2 \times n$ contingency tables and proved that the mixing time is polynomial in $n$ and $N$.
A modified version of the chain on $m \times n$ contingency tables was considered by Chung, Graham, and Yau \cite{ChungGrahamYau}, and they showed that the mixing time is polynomial in $m,n$, and $N$.

Through the adjacency matrix of a given graph, the dynamics of the Diaconis-Gangolli random walk is closely related to the \textit{simple switching} operations of the edges.
Namely, we pick two non-incident edges $i \leftrightarrow j$, $i' \leftrightarrow j'$ uniformly at random at each step, and then we delete the existing two edges and add the edges $i \leftrightarrow j'$, $i' \leftrightarrow j$.
The switch chain formed by these dynamics could be used to generate configuration model random graphs or to generate random graphs with prescribed degree requirements in random networks.
Recently, Tikhomirov and Youssef \cite{TokhomirovYoussef} showed that the switch chain on $d$-regular bipartite graphs on $n$ vertices has a mixing time of order $(nd)^2 \log (nd)$, under the assumption that the degree $d \geq 3$ is at most polynomial in $n$. 
For references in the context of random networks, see, for example, Milo, Kashtan, Itzkovitz, Newman, and Alon \cite{MiloKashtanItzkovitzNewmanAlon}, Rao, Jana, and Bandyopadhyay \cite{RaoJanaBandyopadhyay}, and the survey by Wormald \cite{Wormald}.

Besides the Diaconis-Gangolli model, there have been other algorithms and Markov chains to sample contingency tables.
Dyer, Kannan, and Mount used polytopes to establish an algorithm to sample $m \times n$ contingency tables within polynomial time with respect to $m,n$, and $\log N$.
Morris \cite{Morris} later simplified and improved this algorithm.
Dyer and Greenhill \cite{DyerGreenhill} proved pre-cutoff for a Markov chain on $2 \times n$ contingency tables, known as the heat-bath chain, via the method of path coupling.
This result was later generalized by Matsui, Matsui, and Ono \cite{MatsuiMatsuiOno} to $2 \times 2 \times \cdots \times 2 \times J$ contingency tables, and then by Cryan, Dyer, Goldberg, Jerrum, and Martin \cite{CryanDyerGoldbergJerrumMartin} to contingency tables with a constant number of rows. 

In some other Markov chain models, the cutoff phenomenon has also been widely observed and studied.
Many of the first examples of cutoffs were found in the context of card shuffling, or random walks on the symmetric group. 
For example, the pioneering paper by Diaconis and Shahshahani \cite{DiaconisShahshahani} established cutoff for the random transposition walk on symmetric groups.
Aldous and Diaconis \cite{AldousDiaconis} were the first to use the term ``cut-off phenomenon".
Bayer and Diaconis \cite{BayerDiaconis} proved the cutoff for the riffle shuffle chain, a result later known as ``seven shuffles suffice".

Cutoffs on other models have also been investigated.
For example, Ding, Lubetzky, and Peres \cite{DingLubetzkyPeres} gave a general cutoff result for birth-and-death chains.
Diaconis and Saloff-Coste \cite{DiaconisSaloff-Coste2006} proved a similar result in terms of separation cutoff. 
Lubetzky and Sly \cite{LubetzkySly} found cutoff for random walks on random regular graphs. 
Lubetzky and Peres \cite{LubetzkyPeres} established cutoff for random walks on Ramanujan graphs.
Ganguly, Lubetzky, and Martinelli \cite{GangulyLubetzkyMartinelli} obtained cutoff for the East model.
Numerous other examples could be found in the book by Levin and Peres \cite{LevinPeres}. 

\subsection{Outline}

Section 2 contains a few lemmas that will be useful later. 
We prove a mixing time lower bound for more general random walks on the torus $\bbZ_q^n$ in Section 3, and this result can be applied to both of the specific contingency table random walks.
Section 4 gives detailed proof of the upper mixing time bound for the random walk on $1\times n$ contingency tables over $\bbZ_q$.
We use the same proof strategy in Section 5 to prove an upper mixing time bound for more general random walks on the torus $\bbZ_q^n$, and together with our lower bound from Section 3, we establish cutoff for a class of general random walks on the torus.
Finally, in Section 6, we apply an analogous argument to show cutoff for the random walk on $m \times n$ contingency tables.

\section{Key Lemmas}


\subsection{Distinguishing Statistics}

First, we need the following lemma from \cite{LevinPeres} in order to prove the lower mixing time bounds.

\begin{lemma}[Proposition 7.12, \cite{LevinPeres}]
\label{distinguishingstatistics}
    Let $\mu$ and $\nu$ be two probability distributions on $\Omega$, and let $f$ be a real-valued function on $\Omega$. 
    If
    \begin{align*}
        |\mathbb{E}_\mu(f) - \mathbb{E}_\nu(f)| \geq b \eta
    \end{align*}
    where $\eta^2 := \dfrac{\text{Var}_\mu(f) + \text{Var}_\nu(f)}{2}$, then 
    \begin{align*}
        \| \mu - \nu \|_{\operatorname{TV}} \geq 1 - \dfrac{4}{4 + b^2}.
    \end{align*}
\end{lemma}

\subsection{Technical Lemmas}
\begin{lemma}
\label{summationbound}
For $\alpha \geq 1$ and $n \geq 1$, $n \in \bbN$,
\[
    \sum_{k=1}^n(\alpha n)^{-\frac{k}{k+1}} \leq \frac{1+e^{2/e}}{\sqrt{\alpha}}.
\]
\end{lemma}
\begin{proof}
We break the sum into two parts and bound them, respectively, as follows.
We can write
\[
    \sum_{k=1}^{\lfloor \sqrt{n} \rfloor}(\alpha n)^{-\frac{k}{k+1}} 
    \leq \sqrt{n} (\alpha n)^{-1/2}
    = \frac{1}{\sqrt{\alpha}}
\]
and
\[
    \sum_{k=\lceil \sqrt{n} \rceil}^{n}(\alpha n)^{-\frac{k}{k+1}} 
    \leq (n - \sqrt{n})(\alpha n)^{\frac{1}{1+\sqrt{n}}-1} 
    \leq n(\alpha n)^{\frac{1}{1+\sqrt{n}}-1}
    \leq \frac{n^{\frac{1}{\sqrt{n}}}}{\sqrt{\alpha}} 
    \leq \frac{e^{2/e}}{\sqrt{\alpha}}.
\]
For the last inequality, we used that for $x > 0$, $x^{1/\sqrt{x}}$ is maximized as $e^{2/e}$ at $x = e^2$ by elementary calculus.
Thus
\[
\sum_{k=1}^n(\alpha n)^{-\frac{k}{k+1}} \leq \frac{1}{\sqrt{\alpha}} + \frac{e^{2/e}}{\sqrt{\alpha}} \leq \frac{1+e^{2/e}}{\sqrt{\alpha}}.
\]
\end{proof}

\begin{lemma}
\label{doublesummation}
For all $n \geq 1$, $n \in \bbN$,
    \begin{align*}
        \sum_{\ell=1}^n \sum_{k=1}^n \left(\frac{1}{n^2}\right)^{\frac{\ell}{\ell+1}\frac{k}{k+1}} \leq  1 + 2e^{4/e} + 2e^{8/e} + e^{32/e}.
    \end{align*}
If $\kappa_1 n \leq m \leq \kappa_2 n$ with $0 < \kappa_1 \leq 1 \leq \kappa_2$ and $m \geq 1$, $m \in \bbN$, then
    \begin{align*}
        \sum_{\ell=1}^m \sum_{k=1}^n \left(\frac{1}{mn}\right)^{\frac{\ell}{\ell+1}\frac{k}{k+1}} \leq  \frac{\kappa_2^2(1 + 2e^{4/e} + 2e^{8/e} + e^{32/e})}{\kappa_1}.
    \end{align*}
    \begin{proof}
        We will prove the first bound in the claim, and the second bound will follow as a straightforward corollary.
        For any $n \in \mathbb{N}$, we can write
        \begin{align*}
            \sum_{\ell=1}^n \sum_{k=1}^n \left(\frac{1}{n^2}\right)^{\frac{\ell}{\ell+1} \frac{k}{k+1}} 
             = \sum_{\ell=1}^n \sum_{k=1}^n \frac{1}{n^2} \left(\frac{1}{n^2}\right)^{-\frac{k + \ell + 1}{(k+1)(\ell+1)}} 
            = \frac{1}{n^2} \sum_{\ell=1}^n \sum_{k=1}^n  n^{\frac{2(k + \ell + 1)}{(k+1)(\ell+1)}} 
        \end{align*}
        Note that $n^\frac{2(k + \ell +1)}{(\ell+1)(k+1)}$ is decreasing in both $\ell$ and $k$. 
        We break the sum into four parts.
        First, for $\ell \leq n^{1/4}$ and $k \leq n^{1/4}$,
        \begin{equation}
            \label{bd1}
                    \frac{1}{n^2} \sum_{\ell=1}^{\lfloor n^{1/4} \rfloor} \sum_{k=1}^{\lfloor n^{1/4} \rfloor}  n^{\frac{2(k + \ell + 1)}{(k+1)(\ell+1)}} \leq \frac{1}{n^2} \cdot n^{1/4} \cdot n^{1/4} \cdot n^{3/2} = 1.
        \end{equation}
        Next, for $\ell \geq n^{1/4}$ and $k \geq n^{1/4}$,
        \begin{equation}
        \label{bd2}
        \frac{1}{n^2} \sum_{\ell=\lfloor n^{1/4} \rfloor}^{n} \sum_{k=\lfloor n^{1/4} \rfloor}^{n}  n^{\frac{2(k + \ell + 1)}{(k+1)(\ell+1)}}
        \leq 
        \frac{1}{n^2} \sum_{\ell=\lfloor n^{1/4} \rfloor}^{n} \sum_{k=\lfloor n^{1/4} \rfloor}^{n}  n^{\frac{4(k+\ell)}{\ell k}} 
        \leq \frac{1}{n^2} \cdot n \cdot n \cdot n^{8n^{-1/4}} \leq e^{32/e}
        \end{equation}
        as $x^{8x^{-1/4}}$ is maximized as $e^{32/e}$ at $x = e^4$ by elementary calculus.
        Now it suffices to bound the sum for $\ell \leq n^{1/4}$ and $k \geq n^{1/4}$, and the cases $\ell \geq n^{1/4}$ and $k \leq n^{1/4}$ would follow in an analogous way.
        We consider the scenario of $\ell = 1$ separately.
        When $\ell = 1$, we have
        \begin{equation}
        \label{bd3}
        \frac{1}{n^2} \sum_{k=\lfloor n^{1/4} \rfloor}^{n}  n^{\frac{k+2}{k+1}}
        =
        \frac{1}{n} \sum_{k=\lfloor n^{1/4} \rfloor}^{n}  n^{\frac{1}{k+1}}
        \leq \frac{1}{n} \cdot n \cdot n^{n^{-1/4}} \leq e^{4/e}.
        \end{equation}
        If $2\leq \ell \leq n^{1/4}$, then
        \begin{equation}
        \label{bd4}
        \frac{1}{n^2} \sum_{\ell=2}^{\lfloor n^{1/4} \rfloor} \sum_{k=\lfloor n^{1/4} \rfloor}^{n}  n^{\frac{2(k + \ell + 1)}{(k+1)(\ell+1)}}
        \leq \frac{1}{n^2}\cdot n^{1/4} \cdot n \cdot n^{\frac{2(n^{1/4}+3)}{3n^{1/4}}}
        = n^{-1/12} \cdot n^{2n^{-1/4}} 
        \leq e^{8/e}.
        \end{equation}
        Combining (\ref{bd1}), (\ref{bd2}), (\ref{bd3}), and (\ref{bd4}) establishes the first bound in the claim.

        Now for the second bound, we can write
        \[
            \sum_{\ell=1}^m \sum_{k=1}^n \left(\frac{1}{m n}\right)^{\frac{\ell}{\ell+1}\frac{k}{k+1}}
            \leq
            \sum_{\ell=1}^{\kappa_2 n} \sum_{k=1}^{\kappa_2 n} \left(\frac{\kappa_2^2 n/m}{(\kappa_2 n)^2}\right)^{\frac{\ell}{\ell+1}\frac{k}{k+1}}
            \leq
            \frac{\kappa_2 ^2 n}{m}\sum_{\ell=1}^{\kappa_2 n} \sum_{k=1}^{\kappa_2 n} \left(\frac{1}{(\kappa_2 n)^2}\right)^{\frac{\ell}{\ell+1}\frac{k}{k+1}}.
        \]
        Since $m \geq \kappa_1 n$, $\dfrac{\kappa_2 ^2 n}{m} \leq \dfrac{\kappa_2^2}{\kappa_1}$ and the result follows from the first bound.
    \end{proof}
\end{lemma}

We note that in Lemmas \ref{summationbound}, \ref{doublesummation}, the constants $e^{2/e}, e^{4/e},$ etc. are chosen for convenience to give explicit uniform bounds.

\begin{lemma}
\label{zeroconditionalmean}
    Let $F(x) := \sum_{j = -N}^N \exp(-c(j-x)^2)$, where $N \in \bbN$ and $c > 0$ are constants.
    Then there exists a constant $c_0 > 0$ independent of $N$, such that whenever $c > c_0$, $F(x)$ attains its global maximum at $x = 0$.
\end{lemma}
\begin{proof}
    First, we observe that for any $x \in (-1/2, 1/2]$, we have
    \[
    F(x) > F(x+1) > F(x+2) > F(x+3) > \cdots
    \]
    and
    \[
    F(x) \geq F(x-1) > F(x-2) > F(x-3) > \cdots
    \]
    by comparing the summands in $F(x+k)$ and $F(x+k+1)$, $k \in \bbZ$.
    Thus, it suffices to maximize $F$ on the interval $(-1/2, 1/2]$.

    Let $B_j(\delta) := \{x \in \bbR : |j - x| \leq \delta\}$ for each $j$.
    We look at the second derivative of $F$,
    \[
    F''(x) = \sum_{j = -N}^N 2c(2c(j-x)^2 - 
    1)\exp(-c(j-x)^2).
    \]
    We will only consider $c > 0$ that is not too small.
    If $x\not\in \bigcup_{j = -N}^N B_j(1/\sqrt{2c})$, then $F''(x)$ is necessarily positive, since all the summands above are positive. 
    Combining this with the previous observation, we only need to maximize $F$ on the interval $[-1/\sqrt{2c}, 1/\sqrt{2c}]$.

    We will show that $F'(x) \leq 0$ for $x \in [0,1/\sqrt{2c}]$, which will imply the desired result since $F$ is an even function.
    We can write
    \begin{align*}
        F'(x) 
        &= \sum_{j = -N}^N 2c(j-x)\exp(-c(j - x)^2) \\
        &= -2c \left(xe^{-cx^2} + \sum_{j=1}^N (j+x)e^{-c(j+x)^2} - (j-x)e^{-c(j-x)^2}\right)
    \end{align*}
    and
    we need to show
    \[
    xe^{-cx^2} + \sum_{j=1}^N (j+x)e^{-c(j+x)^2} - (j-x)e^{-c(j-x)^2} \geq 0.
    \]
    Note that for all $c$ not too small and any fixed $x \in [0,1/\sqrt{2c}]$, the function 
    \[
    (z+x)e^{-c(z+x)^2} - (z-x)e^{-c(z-x)^2}
    \]
    is negative and increasing in $z$ for $z \geq 1$. 
    To see that this function is negative, one can apply the Mean Value Theorem to the function $z e^{-cz^2}$ of $z$;
    and to prove that it is increasing, one can apply the Mean Value Theorem to the derivative $\frac{d}{dz}(z e^{-cz^2}) = (1 - 2cz^2)e^{-cz^2}$.
    It follows that
    \begin{align*}
        &\sum_{j=1}^N (j+x)e^{-c(j+x)^2} - (j-x)e^{-c(j-x)^2} \\
        &\geq \sum_{j=1}^\infty (j+x)e^{-c(j+x)^2} - (j-x)e^{-c(j-x)^2} \\
        &\geq (1+x)e^{-c(1+x)^2} - (1-x)e^{-c(1-x)^2} + \int_1^\infty (z+x)e^{-c(z+x)^2} - (z-x)e^{-c(z-x)^2} dz \\
        &= (1+x)e^{-c(1+x)^2} - (1-x)e^{-c(1-x)^2} + \frac{1}{2c}(e^{-c(1+x)^2} - e^{-c(1-x)^2}).
    \end{align*}
    Then it suffices to show that 
    \[
    xe^{-cx^2} + \left(1+x+\frac{1}{2c}\right)e^{-c(1+x)^2} - \left(1-x+\frac{1}{2c}\right)e^{-c(1-x)^2} \geq 0.
    \]
    Note that for all $c$ not too small, the function $g(x) := (1+x+\frac{1}{2c})e^{-c(1+x)^2}$ is convex on the interval $[-1/\sqrt{2c}, 1/\sqrt{2c}]$.
    So for $x \in [0, 1/\sqrt{2c}]$,
    \[
        g(x) - g(-x) 
        \geq 2xg'(-x) 
        = 4cxe^{-c(1-x)^2}\left(x^2-2x+1-\frac{x}{2c}\right).
    \]
    Now, it remains to show that 
    \[
    e^{-cx^2} + 4ce^{-c(1-x)^2}\left(x^2-2x+1-\frac{x}{2c}\right) \geq 0.
    \]
    For $x \in [0,1/\sqrt{2c}]$, we can write
    \begin{align*}
        &e^{-cx^2} + 4ce^{-c(1-x)^2}\left(x^2-2x+1-\frac{x}{2c}\right) \\
        &\geq e^{-cx^2} + 4ce^{-c(1-x)^2}\left(-\frac{x}{2c}\right) \\
        &= e^{-cx^2} - 2xe^{-c(1-x)^2} \\
        &\geq e^{-cx^2} - \sqrt{2/c} \ e^{-c(1-x)^2}.
    \end{align*}
    The last expression above is strictly positive for all $x < \dfrac{1}{2} + \dfrac{\log(c/2)}{4c}$,
    and this completes the proof.
\end{proof}

\section{Lower Bounds}
\subsection{Setup}

We provide a general method for mixing time lower bounds for a family of random walks on the torus $\bbZ_q^n$.
We parameterize the torus $\bbZ_q^{n}$ by $(\bbZ \cap [-q/2, q/2))^{n}$, i.e. these are the representatives we will use for elements in $\bbZ_q^n$.

Let $\{Y_s\}_{s \in \bbN}$ be an arbitrary irreducible symmetric discrete-time Markov chain on $\bbZ^n_q$ with transition matrix $P$, such that $\bbP(Y_s = y_s | Y_{s-1} = y_{s-1}) = \mu(y_s-y_{s-1})$. 
Suppose that $\|Y\|_1 \leq r$ for some real number $r$ in the $L_1$ norm, and $\|Y\|_\infty < q/2$ in the sup norm.
We set $Y_0 = 0$.
Symmetry of $\{Y_s \}$ implies $\mu(y) = \mu(-y)$ for all $y \in \bbZ_q^n$. 
Let $Y \in \bbZ^n$ be a lattice-valued random vector with law $\mu$ mapped to $(\bbZ \cap [-q/2, q/2))^n$. 
We will assume $Y$ has the same marginal variance $\dfrac{\sigma^2}{n}$ for each coordinate.
Write $\operatorname{Var}(Y) = \dfrac{\sigma^2}{n} \Gamma$ where $\Gamma$ is the correlation matrix of $Y$.
Let $\{X_t\}_{t\geq 0}$ be the rate-1 continuous-time Markov chain on $\bbZ_q^n$ with rate matrix $Q = P - I$ and transition kernel $H_t = \exp(tQ)$.
Set $X_0 = 0$.

In this section for lower bounds, we will further assume that as $n \to \infty$, 
\begin{equation}
\label{lowerboundassumptions}
    \sup_{i \neq j} |\Gamma_{ij}| \ll \dfrac{1}{\log(n)} \qquad  \text{and} \qquad \dfrac{r^2}{q^2} \ll \dfrac{1}{\log(n)}.
\end{equation}

Symmetry and irreducibility of $P$ also imply that $H_t$ will converge to the uniform distribution on $\mathbb{Z}_q^n$. 
Additionally, if $f$ is an eigenfunction for $P$ with eigenvalue $\lambda$, then $f$ is also an eigenfunction for $H_t$ with eigenvalue $\exp(t(\lambda - 1))$. 


\subsection{Representations and Eigenvalues}

\begin{definition}
    Let $p$ be any probability measure on $\bbZ_q^{n}$, and let $\rho$ be a representation of $\bbZ_q^{n}$. 
    The \textit{Fourier transform} $\hat{p}$ of $p$ is defined as
    \[
    \hat{p}(\rho) := \sum_{g \in \bbZ_q^{n}} p(g)\rho(g).
    \]
\end{definition}

The irreducible representations of the torus $\bbZ_q^{n}$ are all $1$-dimensional and of the form 
\[
\rho_y(g) = \exp\left(\frac{2 \pi i \langle y, g \rangle}{q}\right), \qquad y \in \bbZ_q^{n}
\]
for $g \in \bbZ_q^{n}$, where $\langle y, g \rangle = \sum_{j=1}^{n-1}g_i y_i$ is the inner product of $y$ and $g$ (see, for example, \cite{NestoridiNguyen}).  

Let $S \subset \bbZ_q^n$ be the symmetric set of generators for $\{Y_m\}$.
Since $\mu(g) = \bbP(Y_1 = g)$ for $g \in S$ and zero elsewhere in $\bbZ_q^n$, then we can write
\begin{align*}
    \hat{\mu}(\rho_y) &= \sum_{g \in \bbZ_q^{n}} \mu(g)\rho_y(g) \\
    &= \sum_{g \in S} \bbP(Y_1 = g)\exp\left(i\left\langle \frac{2 \pi y}{q}, g \right\rangle \right) \\
    &= \Phi_{Y} \left(\frac{2\pi y}{q}\right)
\end{align*}
where $\Phi_{Y}$ is the characteristic function of $Y$.
Note that when put into $\Phi_Y$, $y \in \bbZ_q^n$ should be understood as $y \in (\bbZ \cap [-q/2, q/2))^n$.

By Theorem 3 of \cite{DiaconisShahshahani} (a more detailed proof could be found in Theorem 6 of \cite{Diaconis}), the Fourier transforms $\{\hat{\mu}(\rho_y) : y \in \bbZ_q^{n}\}$ are precisely the eigenvalues of the discrete-time process $\{Y_s\}$, and the characters are the eigenfunctions of $P$.
Thus the eigenvalues $\{\lambda_j : 1\leq j \leq q^{n}\}$ of the continuous-time process $\{X_t\}$ are 
\[
\{\lambda_j : 1\leq j \leq q^{n}\} = \{\exp(\Phi_{Y}(2\pi y/q) - 1) : y \in (\bbZ \cap [-q/2, q/2))^n\}.
\]


Since $\mu$ is symmetric, $\Phi_{Y}$ is a real-valued function. 
Therefore, the real parts of the characters are exactly the eigenfunctions of $P$. 
Hence, for each $y \in \bbZ_q^n$, $f_y(g) := \cos\left(\dfrac{ 2 \pi \langle y, g \rangle}{q}\right)$ is an eigenfunction of $P$ with its corresponding eigenvalue $\Phi_Y(2\pi y/q)$.

\subsection{Mixing Time Lower Bound}

We now aim to apply Lemma \ref{distinguishingstatistics} to obtain a mixing time lower bound. 

\subsubsection{Forming a Distinguishing Statistic}
Let 
\[
\psi_i(x) := \cos\left(\dfrac{2 \pi x_i}{q}\right)
\]
for $(\bbZ \cap [-q/2, q/2))^n$.
We form a distinguishing statistic 
\[
\psi(x) := \sum_{i=1}^n \psi_i(x) = \sum_{i=1}^n \cos\left(\frac{2 \pi x_i}{q}\right)
\]
for $x \in (\bbZ \cap [-q/2, q/2))^n$.
We will use $\psi$ as the real-valued function in Lemma \ref{distinguishingstatistics}.

Note that each $\psi_i$ is an eigenfunction of the transition matrix $P$.
Since $\psi$ is the sum of eigenfunctions whose corresponding eigenvalues are not $1$, we know that 
\begin{equation}
\label{pimean}
    \mathbb{E}_\pi(\psi) = 0
\end{equation} 
(see, for example, Lemma 12.3 in \cite{LevinPeres}). 
Since $\pi$ is uniform over $\mathbb{Z}_q^n$, this implies that each $\psi_i$  is independent of each $\psi_j$ for $i \neq j$ under the stationary distribution. 
Also, for each $i = 1,\dots, n$,
\[
    \mathbb{E}_\pi(\psi_i^2) 
    = \bbE_\pi\left(\frac{1}{2} + \frac{1}{2}\cos\left(\frac{4 \pi x_i}{q}\right) \right) 
    = \frac{1}{2}.
\]
It follows that
\begin{equation}
\label{pivariance}
    \operatorname{Var}_\pi(\psi) = \frac{n}{2}.
\end{equation}

\subsubsection{Mean and Variance under $H_t(0,.)$}
Let $\delta_i$ be the vector whose $i$-th coordinate is $1$ and zero elsewhere for notation. 
We know that each $\psi_i$ is an eigenfunction of $H_t$ with a corresponding eigenvalue of
\[
    \exp\left(t\left(\Phi_{Y}\left(\frac{2 \pi \delta_i }{q}\right) - 1\right)\right) .
\]
By the fourth-order Taylor expansion of the characteristic function $\Phi_{Y}$, we obtain that for any $i$,
\begin{equation}
\label{bound1}
     \exp\left(-\frac{2 \pi^2 t \sigma^2}{n q^2}\right) \leq \exp\left(t\left(\Phi_{Y}\left(\frac{2 \pi \delta_i }{q}\right) - 1\right)\right).    
\end{equation}
The fourth-order Taylor expansion of $\Phi_Y\left(\dfrac{2 \pi (\delta_i + \delta_j)}{q} \right)$ for $i \neq j$ gives
\begin{equation}
\label{bound2}
\begin{aligned}
    &\exp\left(t\left(\Phi_{Y}\left(\frac{2 \pi (\delta_i + \delta_j)}{q}\right) - 1\right)\right) \\
    &\leq \exp\left(t \left(- \dfrac{2 \pi^2 \sigma^2}{nq^2}\left( \delta_i + \delta_j \right)^T \Gamma \left( \delta_i + \delta_j \right) +  \frac{\mathbb{E}((\frac{2 \pi (\delta_i + \delta_j)}{q} \cdot Y)^4)}{24}\right) \right) \\
    &= \exp\left(t\left(-\frac{4 \pi^2 \sigma^2}{nq^2}\left(1 + \Gamma_{ij}\right) + \frac{\mathbb{E}((\frac{2 \pi (\delta_i + \delta_j)}{q} \cdot Y)^4)}{24}\right)\right)
\end{aligned}
\end{equation}
Recall the assumption $\|Y\|_1 \leq r$, so by Cauchy-Schwarz, the above is upper bounded by
\begin{equation*}
\begin{aligned}
    &\leq \exp\left(t\left(-\frac{4 \pi^2 \sigma^2}{nq^2}\left( 1 + \Gamma_{ij} \right) + \frac{\pi^2 r^2 \mathbb{E}((\frac{2 \pi (\delta_i + \delta_j)}{q} \cdot Y)^2)}{6q^2} \right)\right)  \\
    &= \exp\left(t\left(-\frac{4 \pi^2 \sigma^2}{nq^2}\left( 1 + \Gamma_{ij} \right) \left( 1 - \frac{\pi^2 r^2}{3q^2} \right) \right) \right) \\
    &\leq \exp\left(-t\left(\frac{4 \pi^2 \sigma^2}{nq^2}\right) \left( 1 - \sup_{k \neq \ell}|\Gamma_{k \ell} | \right)\left(1 - \frac{\pi^2 r^2}{3q^2}\right)\right).
\end{aligned}
\end{equation*}
By an argument similar to (\ref{bound2}), 
\begin{equation}
    \label{bound3}
    \begin{aligned}        &\exp\left(t\left(\Phi_{Y}\left(\frac{2 \pi (\delta_i - \delta_j)}{q}\right) - 1\right)\right)  \\
    &\leq
    \exp\left(-t\left(\frac{4 \pi^2 \sigma^2}{nq^2}\right) \left( 1 -\sup_{k \neq \ell}|\Gamma_{k \ell} | \right)\left(1 - \frac{\pi^2 r^2}{3q^2}\right)\right).
    \end{aligned}
\end{equation}
Recall that $X_0 = 0$. 
Let $\nu_t$ be the law of $H_t(0, \cdot)$. 
By (\ref{bound1}), we have
\begin{equation}
\label{numeanwitht}
     \mathbb{E}_{\nu_t}(\psi) \geq n  \exp\left(-\frac{2 \pi^2 t \sigma^2}{n q^2}\right).   
\end{equation}
We now set our time to be $t = \dfrac{nq^2 \log(\gamma n)}{4 \pi^2 \sigma^2}$, where $\gamma > 0$ is independent of $n$.  
Plugging this $t$ into (\ref{numeanwitht}), we obtain
\begin{equation}
\label{numean}
    \mathbb{E}_{\nu_t}(\psi) \geq \frac{\sqrt{n}}{\sqrt{\gamma}}.
\end{equation}

Next, we bound the variance of $\psi$ under $\nu_t$ at the time set above. 
We write
\begin{equation}
\label{psivariance}
\begin{aligned}
    \operatorname{Var}_{\nu_t}(\psi)  
    &= \sum_{i=1}^n \operatorname{Var}_{\nu_t}\left(\cos\left(\frac{2 \pi X_t(i)}{q}\right)\right) \\
    &+ \sum_{i=1}^n \sum_{j \neq i} \operatorname{Cov}_{\nu_t}\left(\cos\left(\frac{2 \pi X_t(i)}{q}\right), \cos\left(\frac{2 \pi X_t(j)}{q}\right)\right)
\end{aligned}
\end{equation}
where $X_t(i)$ is the $i$th coordinate of $X_t$. Since cosine is bounded by $1$ in absolute value,
\begin{equation}
\label{varbound}
     \sum_{i=1}^n \operatorname{Var}_{\nu_t}\left(\cos \left(\frac{2 \pi X_t(i)}{q}\right)\right) \leq n.
\end{equation}
It remains to bound the covariance terms in (\ref{psivariance}). 
For any fixed $i,j$, we have by (\ref{bound2}) and (\ref{bound3}) that
\begin{align*}
    &\mathbb{E}_{\nu_t}\left(\cos\left(\frac{2 \pi X_t(i)}{q}\right) \cos\left(\frac{2 \pi X_t(j)}{q}\right)\right) \\
    & = \frac{1}{2}\left(\mathbb{E}_{\nu_t}\left(\cos\left(\frac{2 \pi (X_t(i) + X_t(j))}{q}\right)\right) + \mathbb{E}_{\nu_t}\left(\cos\left(\frac{2 \pi (X_t(i) - X_t(j))}{q}\right)\right)\right) \\
    &\leq \exp\left(-\log(\gamma n)\left(1 - \sup_{k \neq \ell} |\Gamma_{k \ell}| \right) \left( 1 - \frac{\pi^2r^2}{3 q^2}\right)\right) \\
    &= \left(\frac{1}{\gamma n}\right)^{\left(1 - \sup_{k \neq \ell} |\Gamma_{k \ell}| \right) \left( 1 - \frac{\pi^2r^2}{3 q^2}\right) }.
\end{align*}
It follows from (\ref{bound1}) and the above that for any $i$,
\begin{align*}
    & \sum_{j \neq i} \operatorname{Cov}_{\nu_t}\left(\cos \left(\frac{2 \pi X_t(i)}{q}\right), \cos\left(\frac{2 \pi X_t(j)}{q}\right)\right) \\
    & \leq n \left(\left(\frac{1}{\gamma n}\right)^{\left(1 - \sup_{k \neq \ell} |\Gamma_{k \ell}| \right) \left( 1 - \frac{\pi^2r^2}{3 q^2}\right) } - \frac{1}{\gamma n}\right) \\
    &= \gamma^{-1}\left((\gamma n)^{\sup_{k \neq \ell} |\Gamma_{k \ell}| + \frac{\pi^2r^2}{3 q^2} - \sup_{k \neq \ell} |\Gamma_{k \ell}| \frac{\pi^2 r^2}{3 q^2} } - 1\right)
\end{align*}
where the last expression vanishes to $0$ as $n\to \infty$ by the assumptions (\ref{lowerboundassumptions}). 
Hence, for sufficiently large $n$,
\begin{equation}
\label{covbound}
    \sum_{j \neq i} \operatorname{Cov}_{\nu_t}\left(\cos\left(\frac{2 \pi X_t(i)}{q}\right), \cos\left(\frac{2 \pi X_t(j)}{q}\right)\right) \leq 1.
\end{equation}
Putting (\ref{varbound}) and (\ref{covbound}) into (\ref{psivariance}), we obtain
\begin{equation}
\label{nuvariance}
\operatorname{Var}_{\nu_t}(\psi) \leq 2n.
\end{equation}
Now set $b := \dfrac{2}{\sqrt{5\gamma}}$.
With (\ref{pimean}), (\ref{pivariance}), (\ref{numean}), and (\ref{nuvariance}), Lemma \ref{distinguishingstatistics} implies that for sufficiently large $n$,
\begin{align*}
    d(t) \geq 1 - \frac{4}{4 + b^2} = \dfrac{1}{5\gamma + 1}.
\end{align*}
For any $\varepsilon \in (0,1)$, we can set $\gamma = \dfrac{\varepsilon^{-1}-1}{5}$ so that $d(t) \geq \varepsilon$.
With this choice of $\gamma = \gamma(\varepsilon)$, we may conclude that 
\begin{equation}
\label{generallowerbound}
    t_{\operatorname{mix}}(\varepsilon) \geq \frac{nq^2}{4\pi^2 \sigma^2}\left(\log(n) + \log\left(\frac{\varepsilon^{-1}-1}{5}\right)\right)
\end{equation}
for $n$ large enough.

\section{$1\times n$ Contingency Tables}

\subsection{Introduction}

Recall from Section 1 the dynamics of the rate-$1$ continuous-time random walk over the $1\times n$ contingency tables with any fixed coordinate sum.
We will only keep track of the first $n-1$ coordinates so that the resulting process $\{X_t\}_{t \geq 0}$ on $\bbZ_q^{n-1}$ will be aperiodic and irreducible.
We will use $(\bbZ \cap [-q/2, q/2))^{n-1}$ as representatives of $\bbZ_q^{n-1}$.
Let $\{Y_s\}_{s \in \bbN}$ be the embedded discrete-time chain of $\{X_t\}$.
Since our processes are vertex-transitive, without loss of generality, we may assume that $X_0 = 0$ and $Y_0 = 0$.
Denote by $P$ the transition matrix of $\{Y_s\}$.
Then the rate matrix $Q$ for $\{X_t\}$ is $Q = P - I$ where $I$ is the identity matrix, and the transition kernel $H_t$ for $\{X_t\}$ is $H_t = \exp(tQ)$.
We write
\[
H_t(x,y) := \bbP_x(X_t = y) = \bbP(X_t = y | X_0 = x).
\]

Define
\[
S := \{a_{i j} \in \bbZ_q^{n-1}| 1 \leq i ,j \leq n, i\neq j\}
\]
where $a_{i j} \in \bbZ_q^{n-1}$ is defined as follows:
if $1\leq i,j \leq n-1$, then $a_{i j}$ has $1$ on the $i$-th coordinate, $-1$ on the $j$-th coordinate, and zero elsewhere;
if $i = n$, then $a_{ij}$ has $-1$ on the $j$-th coordinate and zero elsewhere;
if $j = n$, then $a_{ij}$ has $1$ on the $i$-th coordinate and zero elsewhere.
Then $\{Y_s\}$ is a random walk on the group $\bbZ_q^{n-1}$ with the symmetric set of generators $S$.

Let $Y$ be a lattice-valued random vector with the same law as that of $Y_1$ mapped to $(\bbZ \cap [-q/2, q/2))^{n-1}$.
Observe that $Y$ has covariance matrix $\dfrac{2}{n} \Gamma$, where 
\[
\Gamma_{ij} = \begin{cases} 1 & i = j \\ -\dfrac{1}{n-1} & \text{else} \end{cases}.
\]
Moreover, it is straightforward to check that 
\[
(\Gamma^{-1})_{ij} =\begin{cases} \dfrac{2(n-1)}{n} & i = j  \\
\\ \dfrac{n-1}{n} & \text{else} \end{cases}.
\]

Since $P$ is aperiodic, irreducible, and symmetric, both $\{X_t\}$ and $\{Y_m\}$ converge to the uniform distribution on $\bbZ_q^{n-1}$ by the convergence theorem of Markov chains.
We will give the mixing time of the continuous-time process $\{X_t\}$, and therefore the mixing time of the $1\times n$ contingency table random walk in continuous time.

\subsection{Eigenvalues for $\{X_t\}$ and $\{Y_s\}$}

By the same reasoning as in Section 3.2, the Fourier transforms $\{\hat{\rho}_y : y \in \bbZ_q^{n-1}\} = \{\Phi_{Y}(2\pi y/q) : y \in (\bbZ \cap [-q/2, q/2))^{n-1}\}$ are precisely the eigenvalues of the discrete-time process $\{Y_s\}$, where $\Phi_{Y}$ is again the characteristic function of $Y$.
And the eigenvalues $\{\lambda_j : 1\leq j \leq q^{n-1}\}$ of the continuous-time process $\{X_t\}$ are 
\[
\{\lambda_j : 1\leq j \leq q^{n-1}\} = \{\exp(\Phi_{Y}(2\pi y/q) - 1) : y \in (\bbZ \cap [-q/2, q/2))^{n-1}\}.
\]

By the continuous-time $\ell^2$-bound,
\begin{equation}
\label{ell2bound}
\begin{aligned}
    4 d(t)^2 &\leq \sum_{j=2}^{q^{n-1}} \lambda_j^{2t} \\
    &= \sum_{\substack{y \in (\bbZ \cap [-q/2, q/2))^{n-1} \\ y \neq 0}} \exp\left(2t\left(\Phi_{Y}\left(\frac{2\pi y}{q}\right) - 1 \right)\right) \\
    &= -1 + \sum_{y \in (\bbZ \cap [-q/2, q/2))^{n-1}}\exp\left(2t\left(\Phi_{Y}\left(\frac{2\pi y}{q}\right) - 1 \right)\right). 
\end{aligned}  
\end{equation}
  One can find a proof of the discrete-time analog of the $\ell^2$-bound in Lemma 12.18 of \cite{LevinPeres}.



\subsection{Strategy for Mixing Time Upper Bound}

We will form the mixing time upper bound via the $\ell^2$ bound in (\ref{ell2bound}). 
We use the link between the eigenvalues of $X_t$ and the characteristic function of the lifted distribution to gain control on the rate of decay of the spectrum of $X_t$. 
This is done by Taylor expanding $\Phi_Y(2 \pi y/q)$ in the regime where $\rVert y \rVert_{\infty} \leq \dfrac{\sqrt{q}}{2 \pi}$. 
In this regime, the eigenvalues of $X_t$ resemble the Gaussian characteristic function and depend on the quadratic form $y^T \Gamma y$, where $\Gamma$ is the correlation matrix of the lifted distribution $Y$. 
In Section 4.5.2., we simplify the quadratic form $y^T \Gamma y$ based on Schur complements of $\Gamma^{-1}$. 
We then use this in Section 4.5.3. to show the $\ell^2$-bound is sufficiently small in the high frequency regime ($\rVert y \rVert_{\infty} \leq \dfrac{\sqrt{ q}}{2 \pi}$). In Section 4.5.1 we also show all contributions to the $\ell^2$-bound in the lower frequency regime ($\rVert y \rVert_{\infty} > \dfrac{\sqrt{q}}{ 2 \pi}$) by directly analyzing the decay of $\Phi_Y$, the characteristic function of our lifted distribution, outside of a neighborhood of zero under the regime where $q \gg n$.

We will first show that the sum in (\ref{ell2bound}) over $y \in (\bbZ \cap [-q/2, q/2))^{n-1}$ with $\| y \|_\infty > \dfrac{\sqrt{q}}{2 \pi }$ is small enough. 
For this to be true, we need to assume that $q \gg n$. 
Specifically, we show that for $y$ with $\| y \|_{\infty} > \dfrac{\sqrt{q}}{2 \pi}$ that $\Phi_{Y}\left(\dfrac{2 \pi y}{q}\right) \leq 1 - \dfrac{1}{25 nq}$. 
For the proof, we set the time to be 
\begin{equation}
\label{mixingtime}
    t = \frac{n q^2\log(\alpha n) }{8 \pi^2}\left(1 - \frac{3}{q}\right)^{-1}
\end{equation}
where $\alpha \geq 1$ is independent of $n$.
Then $\exp(2t(\Phi_{Y}(2 \pi y/q)-1)) $ is of order $n^{-\mathcal{O}(q)}$ at this time $t$, and therefore the sum in (\ref{ell2bound}) over $y$ with $\|y\|_\infty > \dfrac{\sqrt{q}}{2\pi}$ vanishes as $n \to \infty$ with $q \gg n$.

For $y \in (\bbZ \cap [-q/2, q/2))^{n-1}$ with $\| y \|_{\infty} \leq \dfrac{\sqrt{q}}{2 \pi}$, we show that at time (\ref{mixingtime}),
\[
\exp\left(2t\left(\Phi_{Y}\left(\frac{2\pi y}{q}\right) - 1 \right)\right) \leq \exp(- y^T \Gamma y \log(\alpha n)).
\]
Finally, we will use this to conclude that at time $t$, the sum in (\ref{ell2bound}) over $y$ with $\| y \|_{\infty} \leq \dfrac{\sqrt{q}}{2 \pi}$ is bounded above by $\exp\left(\dfrac{4 + 4e^{2/e}}{\sqrt{\alpha}}\right)$ for all $n$ large enough.

\subsection{Bounding Decay in Characteristic Function for Large $\theta$}

We will first show that the characteristic function $\Phi_{Y}(\theta)$ has sufficient decay when $\| \theta \|_\infty > \dfrac{1}{\sqrt{q}}$, and then we will plug in $\theta = \dfrac{2\pi y}{q}$.
We can write the characteristic function as
\begin{equation}
\label{characteristicfunction}
    \Phi_{Y}(\theta) = \frac{ \sum_{j=1}^{n-1} \cos(\theta_j) +  \sum_{j=2}^{n-1}\sum_{k = 1}^{j-1} \cos(\theta_j - \theta_k)}{\binom{n}{2}}
\end{equation}
with $\theta = (\theta_1, \ldots, \theta_{n-1})^T$.

\subsubsection{Regime 1: $\dfrac{1}{\sqrt{q}} < \| \theta \|_\infty \leq \dfrac{\pi}{2}$}
Recall that $Y$ has covariance matrix $\dfrac{2}{n}\Gamma$.
By comparing the second-order terms in the Taylor expansion of $\Phi_{Y}(\theta)$, we have
\begin{equation}
\label{secondorderterm}
     \frac{2\theta^T \Gamma \theta}{n}
     =
     \frac{\sum_{j=1}^{n-1} \theta_j^2 +  \sum_{j=2}^{n-1}\sum_{k = 1}^{j-1} (\theta_j - \theta_k)^2}{\binom{n}{2}} .
\end{equation}
Since 
\begin{equation}
\label{cosinebound}
    \cos(x) \leq 1 - \left(\frac{1}{2} - \frac{\pi^2}{24}\right)x^2 \leq 1 - \frac{2}{25}x^2 \qquad \text{for} \ \  |x| \leq \pi,
\end{equation}
it follows from (\ref{characteristicfunction}) and (\ref{secondorderterm}) that 
\begin{equation}
\label{decay}
    \Phi_{Y}(\theta) \leq 1 - \dfrac{ 4\theta^T \Gamma \theta}{25n} \qquad \text{for} \ \ \rVert \theta \rVert_\infty \leq \dfrac{\pi}{2}.
\end{equation}
Next, note that $\Gamma$ has eigenvalue $1+\dfrac{1}{n-1}$ with eigenspace $\text{span}(\{\mathbbm{1}_{n-1}\})^\perp$ and multiplicity $n-2$, and eigenvalue $\dfrac{1}{n-1}$ with multiplicity $1$ and eigenvector $\mathbbm{1}_{n-1}$. 
Here, $\mathbbm{1}_{n-1}$ is the all-$1$ vector of dimension $n-1$.
Let $\overline{\theta} := \sum_{j=1}^{n-1}\theta_j / (n-1)$ be the average value of the coordinates of $\theta$.
Suppose $\dfrac{1}{\sqrt{q}} < \| \theta \|_\infty \leq \dfrac{\pi}{2}$. 
If $|\overline{\theta}| \leq \dfrac{1}{2\sqrt{q}}$, then
\begin{equation}
\label{decay1}
    \theta^T \Gamma \theta \geq \| \theta - \overline{\theta} \mathbbm{1}_{n-1}\|^2_2 \geq (\|\theta\|_\infty - |\overline{\theta}|)^2 \geq \frac{1}{4q}.
\end{equation}
Otherwise if $|\overline{\theta}| \geq \dfrac{1}{2\sqrt{q}}$, then 
\begin{equation}
\label{decay2}
    \theta^T \Gamma \theta \geq \frac{1}{n-1} \| \bar{\theta} \mathbbm{1}_{n-1} \|^2_2 \geq \frac{1}{4q}.
\end{equation}
Combining (\ref{decay}),(\ref{decay1}), and (\ref{decay2}), we obtain
\begin{equation}
\label{Decay1}
     \Phi_{Y}(\theta) \leq 1 - \frac{1}{25nq}   
\end{equation}
whenever $\dfrac{1}{\sqrt{q}} < \| \theta \|_\infty \leq \dfrac{\pi}{2}$. 

\subsubsection{Regime 2: $\dfrac{\pi}{2} \leq \| \theta \|_\infty \leq \pi$}
We will now show the decay of $\Phi_{Y}(\theta)$ for $\|\theta \|_\infty \geq \dfrac{\pi}{2}$.
Also, we may assume that $\theta \in [-\pi, \pi]^{n-1}$.
For $\theta = (\theta_1,\dots, \theta_{n-1})$, let 
\[
\cA = \cA(\theta) := \{j = 1,\dots, n-1 : |\theta_j| \geq 1\}
\]
with $|\cA|$ being the cardinality of the set $\cA$.

Suppose $|\cA| \leq \dfrac{n-1}{2}$.
Since $\|\theta\|_\infty \geq \pi/2$, there exists $j^* \in \{1,\dots, n-1\}$ such that $|\theta_{j^*}| \geq \pi/2$.
Then 
\begin{equation}
\label{decay3}
    \Phi_{Y}(\theta) = \frac{1}{{\binom{n}{2}}}\left(\sum_{j=1}^{n-1} \cos(\theta_j) + \sum_{j=2}^{n-1}\sum_{k=1}^{j-1} \cos(\theta_j - \theta_k)\right)
    \leq 1 - \frac{\sum_{j \in \cA^c} 1- \cos(\theta_{j^*} - \theta_j)}{{\binom{n}{2}}}.
\end{equation}
Here, we replaced all the cosines in the first sum by $1$, and in the second sum, we only considered the case where $j = j^*$ and $k \in \cA^c$, and replaced all the other cosines by $1$.
Since $|\theta_{j^*}| \geq \pi/2$ and $\theta \in [-\pi, \pi]^{n-1}$,
\[
\pi/2 - 1 \leq |\theta_{j^*} - \theta_j| \leq \pi + 1
\]
for every $j \in \cA^c$.
So (\ref{decay3}) is bounded above by
\begin{equation}
\label{Decay2}
   1 - \frac{\frac{n-1}{2}(1 - \cos(\pi/2 - 1))}{{\binom{n}{2}}} \\
    = 1 - \frac{1 - \cos(\pi/2 - 1)}{n}. 
\end{equation}
     
Now we suppose $|\cA| \geq \dfrac{n-1}{2}$.
Then
\begin{equation}
\label{decay4}
    \Phi_{Y}(\theta) 
    = \frac{1}{{\binom{n}{2}}}\left(\sum_{j=1}^{n-1} \cos(\theta_j) + \sum_{j=2}^{n-1}\sum_{k=1}^{j-1} \cos(\theta_j - \theta_k)\right)
    \leq 1 - \frac{\sum_{j \in \cA} (1 - \cos(\theta_j))}{{\binom{n}{2}}}.
\end{equation}
Here, we replaced all the cosines in the second sum by $1$, and in the first sum, we only considered the case where $j \in \cA$ and replaced all the other cosines by $1$.
Since $\theta \in [-\pi, \pi]^{n-1}$,
\[
-1 \leq \cos(\theta_j) \leq \cos(1)
\]
for every $j \in \cA$.
So (\ref{decay4}) is bounded above by
\begin{equation}
\label{Decay3}
      1 - \frac{\frac{n-1}{2}(1 - \cos(1))}{{\binom{n}{2}}}
     \leq 1 - \frac{1 - \cos(1)}{n}.   
\end{equation}

\subsubsection{Combining the Two Regimes of Large $\theta$}
Now putting together (\ref{Decay1}), (\ref{Decay2}), and (\ref{Decay3}), we get
\begin{equation}
\label{Decay4}
     \Phi_{Y}(\theta) \leq 1 - \frac{1}{25nq}   
\end{equation}
whenever $\theta \in [-\pi, \pi]^{n-1}$ and $\|\theta\|_\infty > \dfrac{1}{\sqrt{q}}$.
Thus for $\|y\|_\infty > \dfrac{\sqrt{q}}{2\pi}$,
\begin{equation}
\label{Decay}
\exp\left(2t\left(\Phi_{Y}\left(\frac{2\pi y}{q}\right) - 1 \right)\right) \leq \exp\left(-\frac{2t}{25nq}\right).
\end{equation}
Plugging in the mixing time in (\ref{mixingtime}) gives
\begin{equation}    \exp\left(2t\left(\Phi_{Y}\left(\frac{2\pi y}{q}\right) - 1 \right)\right) \leq n^{-\dfrac{q}{100\pi^2}}.
\end{equation}
for $\|y\|_\infty > \dfrac{\sqrt{q}}{2\pi}$.
Note that the number of $y \in (\bbZ \cap [-q/2, q/2))^{(n-1)}$ that satisfies $\|y\|_\infty > \dfrac{\sqrt{q}}{2\pi}$ is 
\begin{equation}
\label{elementcount}
    q^{n-1} - \left(\lfloor \frac{\sqrt{q}}{\pi} \rfloor + 1\right)^{n-1} = \cO(q^{n-1}).
\end{equation}
Thus if $q \gg n$, then at the mixing time in (\ref{mixingtime}), we conclude that
\begin{equation}
\label{upperbound1}
    \sum_{\|y\|_\infty > \frac{\sqrt{q}}{2\pi}}\exp\left(2t\left(\Phi_{Y}\left(\frac{2\pi y}{q}\right) - 1 \right)\right) \to 0
\end{equation}
as $n \to \infty$.

\subsection{Bounding Decay in Characteristic Function for Small $\theta$}

By the previous section, we know that the contribution to the sum in (\ref{ell2bound}) over $y$ with $\left\| \dfrac{2 \pi y}{q} \right\|_\infty > \dfrac{1}{\sqrt{q}}$ will be small. 
We will therefore focus on the regime $\left\| \dfrac{2 \pi y}{q} \right\|_\infty \leq \dfrac{1}{\sqrt{q}}$ of $y$.

\subsubsection{Decomposing the Characteristic Function in $\ell^2$-Bound}

By the Taylor expansion of $\Phi_{Y}$, we can write
\begin{equation}
\label{charbound1}
    \Phi_{Y}\left(\frac{2\pi y}{q} \right)
    \leq
    1 - \frac{1}{2}\left(\frac{2\pi y}{q} \right)^T \left(\frac{2}{n}\Gamma \right)\left(\frac{2\pi y}{q} \right)
    +\frac{1}{24}\bbE\left[\left(\frac{2\pi y}{q}\cdot Y\right)^4\right].
\end{equation}
Since $\left\| \dfrac{2 \pi y}{q} \right\|_{\infty} \leq \dfrac{1}{\sqrt{q}}$, i.e. $\|y\|_\infty \leq \dfrac{\sqrt{q}}{2\pi}$, it follows that
\begin{equation}
\label{charbound2}
    \left|\frac{2\pi y}{q}\cdot Y\right| 
    \leq 
    \frac{2\pi \|y\|_{\infty}}{q}\sum_{j=1}^{n-1}|Y(j)|
    \leq \frac{2}{\sqrt{q}}
\end{equation}
as all the coordinates of $Y = (Y(1),\dots, Y(n-1))^T$ are zero except at most two coordinates which take values $\pm 1$.
Combining (\ref{charbound1}) and (\ref{charbound2}), we obtain
\begin{equation}
\begin{aligned}
    \Phi_{Y}\left(\frac{2\pi y}{q} \right)
    &\leq
    1 - \frac{4\pi^2}{nq^2}y^T \Gamma y + \frac{1}{6q}\bbE\left[\left(\frac{2\pi y}{q} \cdot Y \right)^2\right] \\
    &\leq 
    1 - \frac{4\pi^2}{nq^2}y^T \Gamma y \left(1 - \frac{1}{3q} \right).
\end{aligned}
\end{equation}
Therefore, at the mixing time (\ref{mixingtime}),
\begin{align*}
    \exp \left(2t \left(\Phi_{Y}\left(\frac{2 \pi y}{q}\right) - 1 \right)\right)
    \leq \exp(-\log(\alpha n) y^T \Gamma y).
\end{align*}

\subsubsection{Simplifying the Quadratic Form $y^T \Gamma y$}

In this section, we will simplify the quadratic form $y^T \Gamma y$ in the upper mixing time bound. 
For notations, we write $I_k$ for the identity matrix and $J_k$ for the all-$1$ matrix, both of dimension $k \times k$. 
And recall that $\mathbbm{1}_{k}$ is the all-$1$ column vector of dimension $k$. 
We will use the fact that for any $(n-1) \times (n-1)$ symmetric matrix $B = \begin{bmatrix} B_{11} & B_{12} \\ B_{21} & B_{22} \end{bmatrix}$ written in block form, if $B_{11}$ is invertible, then for $x = (x_1, x_2)^T$, 
\begin{equation}
\label{quadratic1}
    x^T B^{-1}x = x_{1}^TB_{11}^{-1}x_{1} + (x_2 - B_{21}B_{11}^{-1}x_1)^T(B_{22} - B_{21}B_{11}^{-1}B_{12})^{-1}(x_2 - B_{21}B_{11}^{-1}x_1).
\end{equation}
One may verify (\ref{quadratic1}) using the decomposition 
\[
B = \begin{bmatrix} I_k & 0 \\ B_{21}B_{11}^{-1} & I_{n-k-1} \end{bmatrix}
\begin{bmatrix} B_{11} & 0 \\ 0 & B_{22} - B_{21}B_{11}^{-1}B_{12} \end{bmatrix}
\begin{bmatrix} I_k & B_{11}^{-1}B_{12} \\ 0 & I_{n-k-1} \end{bmatrix}.
\]
Meanwhile, it is straightforward to check that $\Gamma^{-1} = \dfrac{n-1}{n}(I_{n-1}+J_{n-1})$.
Fix $k = 2,\dots, n-1$.
Let $A(k) \in \mathbb{R}^{k^2}$ be any principal submatrix of $\Gamma^{-1}$ and write
\[
A(k) = \frac{n-1}{n}(I_k + J_k) = \begin{bmatrix} A_{11}(k) & A_{12}(k) \\ A_{21}(k) & A_{22}(k) \end{bmatrix}
\]
with $A_{22}(k) \in \mathbb{R}$. 
Since $\Gamma$ is symmetric, so is $\Gamma^{-1}$, and thus $A_{21}(k) = A_{12}(k)^T$.
Note that 
\begin{equation}
\label{quadratic2}
    A_{11}(k) = \frac{n-1}{n}(I_{k-1} + J_{k-1}), \quad 
A_{12}(k) = \frac{n-1}{n} \mathbbm{1}_{k-1},
\quad
A_{22}(k) = \frac{2(n-1)}{n}.
\end{equation}
We then have 
\begin{equation}
\label{quadratic3}
A_{11}(k)^{-1} = \frac{n}{n-1}\left(I_{k-1} - \frac{1}{k}J_{k-1}\right), \qquad
A_{21}(k) A_{11}(k)^{-1} = \frac{1}{k}\mathbbm{1}_{k-1}^T,
\end{equation}
and
\begin{equation}
\label{quadratic4}
    A_{21}(k) A_{11}(k)^{-1} A_{12}(k) = \frac{(n-1)(k-1)}{nk}.
\end{equation}
Thus by (\ref{quadratic2}) and (\ref{quadratic4}),
\begin{equation}
\label{quadratic5}
A_{22}(k) - A_{21}(k)A_{11}(k)^{-1}A_{12}(k) = \frac{(n-1)(k+1)}{nk}. 
\end{equation}
Now for any $y = (y_1,\dots, y_{n-1})^T$, we can inductively apply (\ref{quadratic1}), (\ref{quadratic3}), and (\ref{quadratic5}) to obtain
\begin{equation}
\label{quadratic}
    y^T \Gamma y = \dfrac{n}{n-1}\left(\dfrac{1}{2}y_1^2 + \sum_{k=2}^{n-1} (y_k - \mu_{k}(y))^2\dfrac{k}{k+1}\right)
\end{equation}
with 
\[
\mu_k(y) := A_{21}(k) A_{11}^{-1}(k)(y_1,\ldots,y_{k-1})^T = \frac{1}{k}\sum_{j=1}^{k-1}y_j
\]
for each $k = 2,3,\dots, n-1$.

\begin{remark}
    As we will show in the next section, the actual value of $\mu_k(y)$ is not significant.
    We will take all the $\mu_k(y)$'s to be zero to obtain an upper bound which is enough for our conclusion.
\end{remark}

\subsubsection{Bounding the Sum for Small $y$}

We now show that the sum in (\ref{ell2bound}) over $y \in \bbZ_q^{n-1}$ with $\left\| \dfrac{2 \pi y}{q} \right\|_\infty \leq \dfrac{1}{\sqrt{q}}$ is small at the mixing time. 
Let 
\[
\cM := \left\{ y \in (\bbZ \cap [-q/2, q/2))^{n-1} : \left\| \dfrac{2 \pi y}{q} \right\|_\infty \leq \dfrac{1}{\sqrt{q}} \right\}
\]
and $N := \lfloor \dfrac{\sqrt{q}}{2 \pi} \rfloor$. 
For $y = (y_1,\dots, y_{n-1})^T$, we can write
\begin{equation*}
\begin{aligned}
    &\sum_{y \in \cM} \exp(-\log(\alpha n)y^T\Gamma y) \\
    &= \sum_{y \in \cM} \exp\left(-\log(\alpha n) \frac{n}{n-1}\left(\frac{1}{2}y_1^2 + \sum_{k=2}^{n-1} (y_k - \mu_k(y))^2\frac{k}{k+1}\right)\right) \\
    &\leq \sum_{y \in \cM} \exp\left(-\log(\alpha n) \left(\frac{1}{2}y_1^2 + \sum_{k=2}^{n-1} (y_k - \mu_k(y))^2\frac{k}{k+1}\right)\right).
\end{aligned}
\end{equation*}
By regrouping the summation with the definition of $\mathcal{M}$, the expression above can be written as
\begin{equation}
\label{small1}
\begin{aligned}
    & \sum_{j_1=-N}^N \exp\left(- \frac{1}{2}\log(\alpha n) j_1^2\right)\left(\sum_{j_2 = -N}^N \exp\left(- \frac{2}{3}\log(\alpha n) (j_2 - \mu_2(y))^2\right)\right)\cdots \\
    & \quad \ldots \left(\sum_{j_{n-1} = -N}^N \exp(- \log(\alpha n) (j_{n-1} - \mu_{n-1}(y))^2)\right).
\end{aligned}
\end{equation}
Next, note that for all $2 \leq k \leq n-1$, $\sum_{j=-N}^N \exp\left(-\dfrac{k+1}{k}\log(\alpha n)(j _k - x)^2\right)$ is maximized at $x = 0$ as long as $n$ is large enough by Lemma \ref{zeroconditionalmean}. 
Thus we obtain
\begin{equation}
\label{small2}
\begin{aligned}
    &\sum_{y \in \cM} \exp(-\log(\alpha n)y^T\Gamma y) \\
    &\leq \left(\sum_{j=-N}^N \exp\left(- \frac{1}{2}\log(\alpha n) j^2 \right)\right) \prod_{k=2}^{n-1}\left(\sum_{j=-N}^N \exp\left(-\frac{k}{k+1}\log(\alpha n)j^2 \right)\right) \\
    &= \exp\left(\sum_{k=1}^{n-1} \log\left(\sum_{j=-N}^N \exp\left(-\frac{k}{k+1}\log(\alpha n)j^2\right)\right)\right). 
\end{aligned}
\end{equation}
We look at the inner expression:
\begin{equation*}
\sum_{j=-N}^N \exp\left(-\frac{k}{k+1}\log(\alpha n)j^2\right) 
= 1 + 2\sum_{j=1}^N \exp\left(-\frac{k}{k+1}\log(\alpha n)j^2\right)
\end{equation*}
Bounding the sum by an integral, we have that the previous expression is at most
\begin{equation}
\label{small3}
\begin{aligned}
& 1 + 2\exp\left(-\frac{k}{k+1}\log(\alpha n)\right) + 2\int_1^N \exp\left(-\frac{k}{k+1}\log(\alpha n)z^2\right) dz \\
&\leq 1 + 2\exp\left(-\frac{k}{k+1}\log(\alpha n)\right) + 2\int_1^\infty z \cdot \exp\left(-\frac{k}{k+1}\log(\alpha n)z^2\right) dz \\
&= 1 + 2\exp\left(-\frac{k}{k+1}\log(\alpha n)\right) + \frac{k+1}{k \log(\alpha n)}\exp\left(-\frac{k}{k+1}\log(\alpha n)\right) \\
&\leq 1 + 4\exp\left(-\frac{k}{k+1}\log(\alpha n)\right).
\end{aligned}
\end{equation}
Using this bound, the fact that $\log(1+x) \leq x$ for all $x > -1$, and Lemma \ref{summationbound}, we have
\begin{equation}
\label{upperbound2}
    \begin{aligned}
    \sum_{y \in \cM} \exp(-\log(\alpha n)y^T\Gamma y)
    &\leq \exp\left(\sum_{k=1}^{n-1} \log\left(1 + 4\exp\left(-\frac{k}{k+1}\log(\alpha n)\right)\right)\right) \\
    &\leq \exp\left(4\sum_{k=1}^{n-1} \exp\left(-\frac{k}{k+1}\log(\alpha n)\right) \right) \\
    &\leq \exp\left(\frac{4 + 4e^{2/e}}{\sqrt{\alpha}}\right).
\end{aligned}
\end{equation}

\subsection{Mixing Time Upper Bound}

We are now ready to state the mixing time upper bound.
For any $\varepsilon \in (0,1)$, set 
\[
\alpha := \inf\left\{z \geq 1 : 4\varepsilon^2 \geq \exp\left(\frac{4 + 4e^{2/e}}{\sqrt{z}}\right) - 1\right\}.
\]
Putting together (\ref{ell2bound}), (\ref{upperbound1}), and (\ref{upperbound2}), we conclude when $t = \dfrac{nq^2 \log(\alpha n) }{8 \pi^2}\left(1 - \dfrac{1}{3q}\right)^{-1}$,
\[
d(t) \leq \varepsilon + o(1)
\]
as $n \to \infty$.

\subsection{Mixing Time Lower Bound}

We will now compute the mixing time lower bound for the $1 \times n$ contingency table walk in continuous time. We showed that the discrete-time increment $Y$ has equivariant coordinates, with each coordinate having a marginal variance of $\dfrac{2}{n}$. 
Recall that the off-diagonal entries of the correlation matrix are equal to exactly $\dfrac{-1}{n-1}$. 
Since at each jump the $1 \times n$ contingency table random walk updates at most two coordinates, each by $\pm 1$, we have that $\|Y\|_1 \leq 2$. And therefore, applying (\ref{generallowerbound}) and setting $r = 2$, $\sigma^2 = \dfrac{2(n-1)}{n}$, we obtain
\[
    t_{\operatorname{mix}}(\varepsilon) \geq \frac{nq^2}{8\pi^2}\left(\log(n) + \log\left(1 - \frac{1}{n}\right) + \log\left(\frac{\varepsilon^{-1}-1}{5}\right)\right)
\]
for all $n$ large enough, assuming $q \gg \sqrt{\log(n)}$.

Now combining Sections 4.6 and 4.7 completes the proof of Theorem \ref{main1}.

\section{Generalization for Mixing Time Upper Bound}

We now generalize our methods for mixing time upper bounds to a family of random walks on tori.
The setup is similar to the one in Section 3 for the general mixing time lower bound.

\subsection{Setup}

Let $\{Y_s\}_{s \in \bbN}$ be an arbitrary irreducible symmetric discrete-time Markov chain on the torus $\bbZ_q^n$ with transition matrix $P$, such that 
\[
\bbP(Y_s = y_s | Y_{s-1} = y_{s-1}) = \mu(y_s - y_{s-1}).
\]
We set $Y_0 = 0$.
Let $Y \in \bbZ^n$ be a lattice-valued random vector with law $\mu$ mapped to $(\bbZ \cap [-q/2, q/2))^n$.
Suppose that $\|Y\|_1 \leq r$, $\|Y\|_\infty < q/2$, and $r^2 \ll q$. 
Note that $Y$ is symmetric.
Assume $Y$ has the same marginal variance $\dfrac{\sigma^2}{n}$ for each coordinate.
Write $\operatorname{Var}(Y) = \dfrac{\sigma^2}{n} \Gamma$ where $\Gamma$ is the correlation matrix of $Y$.
We will assume that $\Gamma$ is positive definite.
Let $\{X_t\}$ be the rate-1 continuous-time Markov chain on $\bbZ_q^n$ with rate matrix $Q = P - I$ and transition kernel $H_t = \exp(tQ)$.

By the same reasoning as in Section 3.2, the eigenvalues of $\{X_t\}$ are 
\begin{equation}
\label{eigenvalues}
    \left\{\exp\left(\Phi_{Y}\left(\frac{2\pi y}{q}\right) - 1\right) : y \in (\bbZ \cap [-q/2, q/2))^n \right\}.
\end{equation}
Recall from (\ref{ell2bound}) that the $\ell^2$-bound says
\[
4d(t)^2 \leq -1 + \sum_{y \in (\bbZ \cap [-q/2, q/2))^{n}}\exp\left(2t\left(\Phi_{Y}\left(\frac{2\pi y}{q}\right) - 1 \right)\right).
\]
We will again upper bound the sum above in two regimes of $y$ separately.

Our strategy in generalizing our upper bound method will be similar to that in Section 4 by analyzing the $\ell^2$ bound. 
We leverage the correspondence between eigenvalues of our walk $X_t$ and the characteristic function of its lift, $\Phi_Y$. Similar to a local central limit theorem, we show our characteristic function $\Phi_Y$ is sufficiently decayed in the low frequency regime $\rVert y \rVert_{\infty} \geq \dfrac{q}{2 \pi}$, and thus our $\ell^2$ sum is controlled. 
We define a Decay Condition that describes the amount of control needed, and briefly discuss tools to show the condition holds. 
In the high frequency regime $\rVert y \rVert_{\infty} \leq \dfrac{\sqrt{q}}{2 \pi}$, we Taylor expand $\Phi_Y$ and leverage the correlation structure of the lifted distribution $Y$. 
We form a Correlation Condition based on $\Gamma$ to gain control on the rate of decay of the spectrum of $X_t$.

\subsection{Regime 1: $\|y\|_\infty > \dfrac{\sqrt{q}}{2\pi}$}

When $y$ is big, we need $Y$ to satisfy the following decay condition, so that we can show that our $\ell_2$ sum is very small in the low frequency regime in a format analogous to a local central limit theorem. This condition is based on the rate of decay of the characteristic function $\Phi_Y$ outside of a neighborhood of zero.
For this condition to hold, we may need some additional asymptotic assumptions on $q$.
\begin{definition}
    Let 
    \[
    \cL := \left\{ y \in (\bbZ \cap [-q/2, q/2))^{n} : \left\| \dfrac{2 \pi y}{q} \right\|_\infty > \dfrac{1}{\sqrt{q}}\right\}.
    \]
    and let $Y$ be a symmetric random variable supported on $(\bbZ \cap [-q/2, q/2))^n$ with marginal variance $\dfrac{\sigma^2}{n}$ by coordinate.
    We say that $Y$ satisfies the \textit{Decay Condition} if for
    \[
    t \geq \dfrac{nq^2 \log(n)}{4 \pi^2 \sigma^2},
    \]
    the sum 
    \[
    \sum_{y \in \cL}\exp\left(2t\left(\Phi_{Y}\left(\frac{2\pi y}{q}\right) - 1 \right)\right) \to 0
    \]
    as $n \to \infty$, where $\Phi_Y$ is the characteristic function of $Y$.
\end{definition}
Note that we required $q \gg n$ in the $1 \times n$ contingency table case for this condition to hold.
We give a few other examples in which the Decay Condition is met.
\begin{example}[Simple Random Walk on $\bbZ_q^n$]
    If $\{Y_s\}_{s\in \bbN}$ is the simple random walk on $\bbZ_q^n$, then $Y$ is the increment for the simple random walk on $\bbZ^n$.
    In other words, $Y$ takes values in 
    \[
    \{(0,\dots, 0, \pm 1, 0\dots, 0) \in \bbZ^n\},
    \]
    each with probability $\dfrac{1}{2n}$.
    In this case, $r = \sigma^2 = 1$.
    The characteristic function $\Phi_Y(\theta) = \dfrac{1}{n}\sum_{k=1}^n \cos(\theta_k)$ for $\theta = (\theta_1,\dots, \theta_n) \in \bbR^n$.
    Using the bound (\ref{cosinebound}) for cosine, we have that for $y = (y_1,\dots, y_n) \in \cL$ and $t \geq \dfrac{nq^2 \log(n)}{4 \pi^2}$,
    \begin{equation}
    \label{SRWexample}
        \begin{aligned}
        \exp\left(2t\left(\Phi_{Y}\left(\frac{2\pi y}{q}\right) - 1 \right)\right) 
        &\leq 
        \exp\left(2t\left(-1 + \frac{1}{n}\sum_{k=1}^n \cos\left(\frac{2\pi y_k}{q}\right)\right)\right) \\
        &\leq 
        \exp\left(2t\left(-1 + \frac{1}{n}\sum_{k=1}^n 1 - \frac{2}{25}\frac{4\pi^2 y_k^2}{q^2}\right)\right) \\
        &=
        \exp\left(-2t\left(\sum_{k=1}^n \frac{2}{25n}\frac{4\pi^2 y_k^2}{q^2}\right)\right)
    \end{aligned}
    \end{equation}
For $y \in \mathcal{L}$, $\|y\|_\infty  > \dfrac{\sqrt{q}}{2\pi}$.
Since $\sum_{k=1}^n y_k^2 \geq \|y\|_\infty^2$, the expression above can be upper bounded by
    \begin{equation*}
        \exp\left(2t\left(-\frac{2}{25n}\frac{4\pi^2 \|y\|_\infty^2}{q^2}\right)\right) 
        \leq
        \exp\left(2t\left(-\frac{2}{25nq}\right)\right) 
        \leq
        n^{-\dfrac{q}{25\pi^2}}.
    \end{equation*}
    Similar to (\ref{elementcount}), the cardinality of $\cL$ is $\cO(q^n)$.
    Thus the Decay Condition is met as long as $q \gg n$.
\end{example}

\begin{example}[Dirichlet Form Comparison]
Methods of Dirichlet form comparison can be used to find regimes with the Decay Condition. 
\begin{definition}
    Let $f$ be a real-valued function on a state space $\Omega$, and let $P$ be a reversible and irreducible transition matrix on $\Omega$ with stationary distribution $\pi$. 
    We define the \textit{Dirichlet form},
    \begin{align*}
        \mathcal{E}(f) := \langle (I-P)f, f \rangle_{\pi}
    \end{align*}
    where the inner product is defined as
    \begin{align*}
        \langle f, g \rangle_{\pi} := \sum_{x \in \Omega} \pi(x) f(x) g(x)
    \end{align*}
\end{definition}
For any non-constant eigenfunction $f$ of $P$ with  $\operatorname{Var}_\pi(f) = 1$, we have that
\begin{equation}
    \label{dirichleteigen}
    \mathcal{E}(f) = 1 - \lambda
\end{equation}
with $\lambda$ being the corresponding eigenvalue of $f$. 

Using this, let $P$ be the transition matrix for $\{Y_s \}$ and $\widetilde{P}$ be the transition matrix for the simple random walk on $\bbZ_q^n$. Suppose we have a constant $B = B(n)$ so that for any real-valued function $g$ on $\Omega$ we have
\begin{equation}
    \label{pathmethod}
    \begin{aligned}
        \widetilde{\mathcal{E}}(g) \leq B \mathcal{E}(g)
    \end{aligned}
\end{equation}
If we set $g(x) = \cos\left(\dfrac{2 \pi \langle y, x \rangle}{q}\right)$, for $\| y \|_\infty \geq \dfrac{\sqrt{q}}{2 \pi}$ we obtain by (\ref{SRWexample}), (\ref{dirichleteigen}), and (\ref{pathmethod}),
\begin{equation}
    \label{pathbound}
    \begin{aligned}
        \exp\left(2t\left( \Phi_Y\left(\frac{2 \pi y}{q}  \right) - 1 \right)\right) = \exp\left( -2t \mathcal{E}(g)  \right) \leq \exp \left( - \dfrac{2t \tilde{\mathcal{E}}(g)}{B} \right) \leq n^{-\dfrac{q }{25 \pi^2 B \sigma^2 }}
    \end{aligned}
\end{equation}
Therefore, since the cardinality of $\mathcal{L}$ is $\mathcal{O}(q^n)$, the Decay Condition is met if $\dfrac{q}{\log q} \gg \dfrac{Bn \sigma^2}{\log n}$.
\begin{remark}
    Several studied methods use Dirichlet form comparison to form such a constant $B$. 
    For example, \cite{DiaconisSaloff-Coste} establishes the method of canonical paths and forms $B$ by congestion ratio between the dynamics of $P$ and $\widetilde{P}$. 
    Chapter 13 of \cite{LevinPeres} gives more details on the subject. Additionally, \cite{Benedicks} derives an upper bound to the decay of the characteristic functions for integer-valued random walks by also considering paths that lead to $1 \in \bbZ$.
\end{remark}
\end{example}

\begin{example}[Irreducible Lifted Walk on $\bbZ^n$]
Let $\{Y_s\}_{s \in \bbN}$ again be an irreducible symmetric random walk on $\bbZ_q^n$ as in the setup above in Section 5.1, and let $Y$ be defined accordingly.
Let $\{S_s\}_{s \in \bbN}$ be a random walk on $\bbZ^n$ with increment $Y$.
First, we observe that (\ref{eigenvalues}) with $t=1$ corresponds to the characteristic function for a random variable on $\mathbb{Z}^n$ with probability mass function $p$ given by
\[
    p(x) = e^{-1} \sum_{s=0}^\infty  \dfrac{\mathbb{P}(S_s = x)}{s!}. 
\]
By the irreducibility of $S_s$, we know that the random walk on $\mathbb{Z}^n$ with incremental distribution $p$ is aperiodic and irreducible. 
Therefore, by Lemma 2.3.2 of \cite{LawlerLimic}, for all $\theta \in [-\pi, \pi]^n$, 
\[
\exp(\Phi_{Y}(\theta) - 1)\leq \exp(-b\|\theta\|_2^2)
\]
for some $b > 0$.
Here, $b$ depends on $n$ and the distribution of $Y$.
When $\dfrac{q}{\log q} \gg \dfrac{\sigma^2}{b\log(n)}$ as $n \to \infty$, the Decay Condition is met for $Y$.
\end{example}

\subsection{Regime 2: $\|y\|_\infty \leq \dfrac{\sqrt{q}}{2\pi}$}

On the other hand, when $y$ is small, we will assume the following correlation condition on the correlation matrix $\Gamma$ of $Y$. Our correlation condition depends on the magnitude of Schur complements of $\Psi$, the inverse of the random walk's correlation matrix $\Gamma$.
Schur complements of $\Psi$ have a direct connection to the conditional variances of the multivariate Gaussian distribution with correlation matrix $\Gamma$. 
This connection to the multivariate Gaussian distribution arises from studying the lifted random walk of $X_t$, which then is linked to the local limit theorem in our regime when $q$ is large.
Our condition checks the magnitude of conditional variances for the multivariate Gaussian distribution with correlation matrix $\Gamma$. We will now state the condition and afterwards provide more interpretation.

\begin{definition}
Let $\Gamma_n$ be a sequence of correlation matrices such that $\Gamma_n \in \mathbb{R}^{n \times n}$. 
We say $(\Gamma_n)_{n=1}^\infty$ meets the \textit{Correlation Condition} if both of the following conditions hold. 

\begin{enumerate}[(i)]
    \item Let $\Psi_n := \Gamma_n^{-1}$.
    Then $\sup_n \rVert \Psi_n \rVert_\infty < \infty$. 

    \item Let $\Psi_{n,k}$ be the principal submatrix obtained by taking the first $k$ rows and columns of $\Psi_n$. 
    Further partition
    \begin{align*}
        \Psi_{n,k} = \begin{bmatrix} \Psi_{n,k}^{(11)} & \Psi_{n,k}^{(12)} \\
        \Psi_{n,k}^{(21)} & \Psi_{n,k}^{(22)} \end{bmatrix}
    \end{align*}
    where $\Psi_{n,k}^{(11)}$ is the submatrix attained by taking the first $k-1$ rows and columns of $\Psi_{n,k}$, and $\Psi_
    {n,k}^{(22)}$ is the $(k,k)$ entry of $\Psi_{n,k}$. 
    Let $a_1^{(n)} := \Psi_{n,1}^{-1}$, and for $k \geq 2$, let
    \begin{align*}
        a_k^{(n)} := \left( \Psi_{n,k}^{(22)} -  \Psi_{n,k}^{(21)} \left(\Psi_{n,k}^{(11)}\right)^{-1} \Psi_{n,k}^{(12)} \right)^{-1}.
    \end{align*}
    Then for any $\alpha \geq 1$,
    \begin{align*}
        \sup_n \sum_{k=1}^n \left(\frac{1}{\alpha n }\right)^{a_k^{(n)}} \leq g(\alpha)
    \end{align*}
    where $g$ is a monotone decreasing function such that $\lim_{\alpha \rightarrow \infty} g(\alpha) = 0$, and $g$ does not depend on $n$.
\end{enumerate}
\end{definition}

As previously mentioned, the Correlation Condition depends on conditional variances of the multivariate Gaussian distribution. First, note that for $1 \leq i \leq n$, the $i$th diagonal entry of $\Psi$ can be written as $ \operatorname{Var}(Z_i \mid Z_1, \ldots, Z_{i-1}, Z_{i+1}, \ldots, Z_n)^{-1}$, where $Z$ has the multivariate Gaussian distribution with correlation matrix $\Gamma$, assuming unit variances. Thus, the condition first checks that the variance of one coordinate conditioned on the rest stays bounded above from zero as the dimension $n$ grows to infinity. 
Moreover, the elements $a^{(n)}_k$ are equal to the conditional variances $\operatorname{Var}(Z_k \mid Z_{k+1}, \dots Z_{n})$, where $Z$ again has the multivariate Gaussian distribution with correlation matrix $\Gamma$ and unit variances. 
The condition therefore checks that conditional variance of the multivariate Gaussian distribution with correlation matrix $\Gamma$ also decays from $1$ sufficiently slow as we condition on more and more variables.

\begin{remark}
  This Correlation Condition is always met for uncorrelated random walks, specifically with $g(\alpha) = 1/\alpha$.  
\end{remark}

Now for $\| y \|_\infty \leq \dfrac{\sqrt{q}}{2 \pi}$, we can bound the characteristic function as
\begin{align*}
    \Phi_{Y}(\theta) \leq 1 - \frac{1}{2} \left( \frac{2 \pi y}{q} \right)^T \left( \frac{ \sigma^2}{n} \Gamma \right) \left( \frac{2 \pi y}{q} \right) + \frac{1}{24} \mathbb{E}\left( \left( \frac{2 \pi y }{q} \cdot Y \right)^4 \right).
\end{align*}
As before, we can write
\begin{align*}
    \left| \frac{2 \pi y}{q} \cdot Y \right| & \leq \frac{2 \pi \| y \|_{\infty}}{q} \sum_{j=1}^n |Y(j)| \leq \dfrac{r}{\sqrt{q}}   
\end{align*}
where $Y(j)$ is the $j$-th coordinate of $Y$.
Thus
\begin{align*}
    \Phi_{Y}\left(\frac{2 \pi y}{q}\right) &\leq 1 - \frac{1}{2} \left( \frac{2 \pi y}{q} \right)^T \left( \frac{ \sigma^2}{n} \Gamma \right) \left( \frac{2 \pi y}{q} \right) + \frac{1}{24} \mathbb{E}\left( \left( \frac{2 \pi y }{q} \cdot Y \right)^4 \right) \\
    & \leq 1 - \frac{1}{2} \left( \frac{2 \pi y}{q} \right)^T \left( \frac{ \sigma^2}{n} \Gamma \right) \left( \frac{2 \pi y}{q} \right) \left(1 - \frac{r^2}{12q}\right).
\end{align*}
Now, at the mixing time
\begin{equation}
\label{generalmixingtime}
    t = \dfrac{nq^2 \log(\alpha n)}{4 \pi^2 \sigma^2}\left(1 - \dfrac{r^2}{12q}\right)^{-1}
\end{equation}
where $\alpha \geq 1$ is independent of $n$, we have
\begin{equation}
\label{eigenvaluestoquadraticform}
\begin{aligned}
    & \exp\left(2t\left(\Phi_{Y}\left(\frac{2 \pi y}{q}\right) - 1\right)\right)\\
    & \leq \exp\left(\dfrac{nq^2 \log(\alpha n)}{2 \pi^2 \sigma^2}\left(1 - \frac{r^2}{12q}\right)^{-1} \frac{1}{2} \left( \frac{2 \pi y}{q} \right)^T \left( \frac{ \sigma^2}{n} \Gamma \right) \left( \frac{2 \pi y}{q} \right)\left(1 - \frac{r^2}{12q}\right)\right) \\
    &= \exp(-\log(\alpha n) y^T \Gamma y).
\end{aligned}
\end{equation}

To complete the case of $\|y\|_\infty \leq \dfrac{\sqrt{q}}{2\pi}$, we use the following lemma, a generalization of Section 4.5.3.
\begin{lemma}
\label{upperboundsmally}
    Let 
\[
\cM := \left\{ y \in (\bbZ \cap [-q/2, q/2))^{n} : \left\| \dfrac{2 \pi y}{q} \right\|_\infty \leq \dfrac{1}{\sqrt{q}}\right\}.
\]
    If the correlation matrix $\Gamma$ of $\xi_1$ satisfies the Correlation Condition with the bounding function $g(\alpha)$, then
    \[
    \sum_{y \in \cM} \exp(-\log(\alpha n) y^T \Gamma y)
    \leq \exp\left(\left(2 + \sup_n \|\Psi_n\|_\infty \right)g(\alpha)\right)
    \]
    for all $n \in \bbN$ large enough.
\end{lemma}
\begin{proof}
    The proof is similar to the one presented in Section 4.5.3.
    Again, let $N := \lfloor \dfrac{\sqrt{q}}{2\pi} \rfloor$.
    The analogous calculation as in (\ref{small1}) and (\ref{small2}) still holds, and from that we obtain that for $y = (y_1,\dots, y_{n-1})^T $,
    \[
    \sum_{y \in \cM} \exp(-\log(\alpha n)y^T\Gamma y) \leq \exp\left(\sum_{k=1}^{n} \log\left(\sum_{j=-N}^N \exp\left(-a_k^{(n)}\log(\alpha n)j^2\right)\right)\right). 
    \]
    Since 
    \[
    \frac{1}{a_k^{(n)}} = \Psi_{n,k}^{(22)} -  \Psi_{n,k}^{(21)} \left(\Psi_{n,k}^{(11)}\right)^{-1} \Psi_{n,k}^{(12)} \leq \sup_n \|\Psi_n\|_\infty < \infty,
    \]
    a similar calculation as in (\ref{small3}) gives
    \[
    \sum_{j=-N}^N \exp\left(-a_k^{(n)}\log(\alpha n)j^2\right) \leq 1 + (2 + \sup_n \|\Psi_n\|_\infty)\exp\left(-a_k^{(n)}\log(\alpha n)\right).
    \]
    Now, as in (\ref{upperbound2}), it follows that
    \begin{align*}
    \sum_{y \in \cM} \exp(-\log(\alpha n)y^T\Gamma y)
    &\leq 
    \exp\left(\left(2 + \sup_n \|\Psi_n\|_\infty \right)\sum_{k=1}^n \left(\frac{1}{\alpha n }\right)^{a_k^{(n)}}\right) \\
    &\leq \exp\left(\left(2 + \sup_n \|\Psi_n\|_\infty \right)g(\alpha)\right).
    \end{align*}
\end{proof}

\subsection{Cutoff for a Family of Random Walks on $\bbZ_q^n$}

We now summarize our generalization in Sections 3 and 5 with the complete statement of Theorem 3.

\setcounter{theorem}{2}
\begin{theorem}
\label{main2}
Let $\{ X_t \}^{(n)}$ be an arbitrary family of irreducible symmetric continuous-time Markov chains on the torus $\bbZ_q^n$ with the setup in Sections 3.1 and 5.1. 
Specifically, we define the increment of each jump mapped to $(\bbZ \cap [-q/2,q/2))^n$, $Y^{(n)}$, to have a correlation matrix $\Gamma_n$, and require $\| Y^{(n)} \|_1 \leq r = r(n)$, and $\|Y^{(n)}\|_\infty < q/2$.
Suppose that $Y^{(n)}$ has the same marginal variance by coordinate of $\dfrac{\sigma^2}{n}$, where $\sigma^2 = \sigma^2(n)$.
Let $\varepsilon \in (0,1)$ be arbitrary. 
Then we have the following
\begin{enumerate}[(a)]
    \item 
    Suppose that $\{ Y^{(n)} \}$ meets the Decay Condition and $\{ \Gamma_n \}$ meets the Correlation Condition with the bounding function $g(\alpha)$. 
    Let
    \[
    \alpha := \inf\left\{z \geq 1 : 4\varepsilon^2 \geq \exp\left( \left( 2 + \sup_n \rVert \Psi_n \rVert_\infty \right) g(z)\right) - 1\right\}.
    \]
    When $t = \dfrac{nq^2 \log(\alpha n) }{4 \pi^2 \sigma^2}\left(1 - \dfrac{r^2}{12q}\right)^{-1}$,
    \[
    d^{(n)}(t) \leq \varepsilon + o(1)
    \]
    as $n \to \infty$, $q \rightarrow \infty$, with $r^2 \ll q$.
    
    \item 
    On the other hand, as $n \to \infty$ and $q \to \infty$ with $\sup_{k \neq \ell} |\Gamma_{n, k \ell} | \ll \dfrac{1}{\log(n)}$, and $\dfrac{r^2}{q^2} \ll \dfrac{1}{\log(n)}$
    \[
    t^{(n)}_{\operatorname{mix}}(\varepsilon) \geq \frac{nq^2}{4\pi^2 \sigma^2}\left(\log(n) + \log\left(\frac{\varepsilon^{-1}-1}{5}\right)\right).
    \]
\end{enumerate}    
\end{theorem}
\begin{remark}
    We see by Theorem 3 any family of Markov chains $\{X_t\}^{(n)}$ in the setup of Section 5.1 meeting the Decay and Correlation Conditions with $n \rightarrow \infty$, $q \rightarrow \infty$, $\dfrac{r^2}{q^2} \ll \dfrac{1}{\log(n)}$, and $\sup_{k \neq \ell} |\Gamma_{n, k \ell} | \ll \dfrac{1}{\log(n)}$ exhibits cutoff at the time $\dfrac{nq^2 \log(n)}{4 \pi^2 \sigma^2}$, with a cutoff window of order $\dfrac{nq^2}{\sigma^2}$.
\end{remark}
\begin{remark}
    \label{cutoffremark}
    Let $\{ (S_{t}^{(n)})_{t \geq 0} \}_{n=1}^\infty$ be a family of symmetric and irreducible random walks on $\mathbb{Z}^n$ in continuous time with bounded increments, such that each coordinate is equivariant with marginal variance $\dfrac{\sigma^2}{n}$, and $S_{t}^{(n)}$ meets the correlation condition. 
    Then by the same reasoning as Example 3, we can set $q = q(n)$ large enough so that $S_{t}^{(n)} \mod q$ has cutoff at time $t = \dfrac{nq^2 \log(n)}{4 \pi^2 \sigma^2}$, assuming $\sup_{k \neq \ell} |\Gamma_{n, k \ell} | \ll \dfrac{1}{\log(n)}$.
\end{remark}

\subsection{Discussion of the Regime $q \rightarrow \infty$}

In the case of our main theorems, we require $q \gg n$ or $q \gg n^2$ for the upper bounds in order to show sufficient decay in the characteristic function of our lifted distribution $\Phi_Y$ to then bound the $\ell^2$ summation. 
This was done as a tool to gain control of the spectrum: they are sufficient (but not necessarily sharp) assumptions for the Decay Conditions to be met in the contingency table cases.
In general, the requirement of the regime of $q$ depends on how one fulfills the Decay Condition for their specific model.
For example, Nestoridi and Nguyen \cite{NestoridiNguyen} showed cutoff for the Diaconis and Gangolli walk on $n \times n$ contingency tables over $\bbZ_q$ when $q$ is small compared to $n$, and \cite{hermon2018supplementary} gives results for the simple random walk on $\bbZ_q^n$.
Hermon and Olesker-Taylor in \cite{hermon2018supplementary} also analyzed the cutoff times for random walks on the torus $\mathbb{Z}_q^n$ when $q$ is small.

For the mixing time lower bounds, the assumptions on $q$ are relatively mild and are primarily for the Taylor expansion analysis in the distinguishing statistics. 

We do like to point out that there are cases where a fast-growing $q$ relative to $n$ is required for cutoff. 
In the regime where $q$ is small, we are not necessarily able to compare the rate of decay of the spectrum of our walk to the Gaussian characteristic function.
Consider the following rate-$1$ random walk $X_t = (X_t^{(1)}, X_t^{(2)},\dots , X_t^{(n)})$ on the torus $\bbZ_q^n$.
For $j \geq 3$, each $X_t^{(j)}$ is a rate-$(1/n)$ simple random walk on $\bbZ_q$, independent of all the other coordinates, including $X_t^{(1)}$ and $X_t^{(2)}$.
Each of $X_t^{(1)}$ and $X_t^{(2)}$ performs simple random walk jumps with rate $e^{-n}/n$, independent of each other and all the remaining $n-2$ coordinates. 
In addition, with rate $(1-e^{-n})/n$, $X_t^{(1)}$ and $X_t^{(2)}$ jump simultaneously but with independent simple random walk increments. 

This dynamics has correlation matrix $I_n$, the $n \times n$ identity matrix, bounded increments, and is irreducible over the torus $\mathbb{Z}_q^n$. 
Since $X_t$ is symmetric, it converges to the uniform distribution $\pi$ on $\bbZ_q^n$. By Remark \ref{cutoffremark}, as $n \to \infty$, there is a regime of $q = q(n)$ sufficiently large with respect to $n$, so that $X_t$ has cutoff at the time $\dfrac{n q^2 \log(n)}{4 \pi^2}$ with a cutoff window of order $\mathcal{O}(n q^2)$.

The walk does not mix with cutoff in the regime when $n \rightarrow \infty$ and $q$ is even. 
In this regime, the difference of the first two coordinates $X_t^{(1)}$ and $X_t^{(2)}$ modulo $2$ is the primary bottleneck in the chain's trajectory.
To see this, let $D_t := X_t^{(1)} - X_t^{(2)} \mod 2$, which is a rate-$(2e^{-n}/n)$ simple random walk on $\bbZ_2$ as a result of $q$ being even. 
We may assume $X_0 = 0$ due to transitivity of $\bbZ_q^n$, and let $H_t$ be the transition kernel of $X_t$.
A direct calculation with the Kolmogorov forward and backward equations yields
\[
\|H_t(0, \cdot) - \pi\|_{\operatorname{TV}} \geq \bbP(D_t = 0) - \frac{1}{2} = \frac{1}{2} \exp(-4te^{-n}/n).
\]
Thus, for $\varepsilon < 1/2$, 
\begin{equation}
    \label{examplelowerbd}
    t^{(n)}_{\operatorname{mix}}(\varepsilon) \geq \frac{n e^n \log(\frac{1}{2\varepsilon})}{4}.
\end{equation}

We now give an upper bound to $t_{\operatorname{mix}}(1 - \varepsilon)$ for $\varepsilon < 1/q^2$. 
The following lemma can be proved directly using the optimal coupling definition of the total variation distance (see Proposition 4.7 of \cite{LevinPeres}).
\begin{lemma}
    \label{tvproduct}
    Let $\mu_1$ and $\mu_2$ be probability measures supported on the discrete outcome set $\Omega$ with $\rVert \mu_1 - \mu_2 \rVert_{\operatorname{TV}} = \varepsilon_1$, and let $\nu_1$ and $\nu_2$ be probability measures supported on the discrete outcome set $\Omega'$ with $\rVert \nu_1 - \nu_2 \rVert_{\operatorname{TV}} = \varepsilon_2$. The product measures $\mu_1 \times \nu_1$ and $\mu_2 \times \nu_2$ satisfy
\[
        \rVert \mu_1 \times \nu_1 - \mu_2 \times \nu_2 \rVert_{\operatorname{TV}} \leq \varepsilon_1 + \varepsilon_2 - \varepsilon_1 \varepsilon_2.
\]
\end{lemma}

We will view $X_t$ as the product chain of $(X_1, X_2)$ and each of the remaining coordinates.
It is straightforward to show that at any $t \geq 0$, the law of $(X_t^{(1)}, X_t^{(2)})$ has total variation distance at most $1-1/q^2$ with the uniform distribution of $\bbZ_q^2$.
For each $j = 3, \dots, n$, $X_t^{(j)}$ is a rate-$1/n$ simple random walk on $\bbZ_q$.
It is well-known that the spectral gap of the discrete-time transition matrix of $X_t^{(j)}$ is $\gamma := 1 - \cos(2\pi/q)$ (see, for example, 12.3.1 in \cite{LevinPeres}).
Then, by Theorem 20.6 in \cite{LevinPeres}, the total variance distance between the transition kernel of $X_t^{(j)}$ at time $t$ and the uniform distribution on $\bbZ_q$ is at most 
\[
q \cdot \exp(-\gamma t/n).
\]
Therefore, at time 
\[
t^\star = \frac{n\log(nq (1/q^2 - \varepsilon)^{-1})}{1 - \cos(2\pi/q)},
\]
the total variance distance between the transition kernel of $X_t^{(j)}$ and the uniform distribution on $\bbZ_q$ is at most $\frac{1/q^2 - \varepsilon}{n}$.
Now by Lemma \ref{tvproduct}, 
\[
d(t^\star) \leq 1 - \frac{1}{q^2} + \sum_{j = 3}^n \frac{1/q^2 - \varepsilon}{n} \leq 1 - \varepsilon.
\]
It follows that $t_{\operatorname{mix}}(1 - \varepsilon) \leq  t^\star$.
Combining this with (\ref{examplelowerbd}), we have that for $\varepsilon < 1/q^2$
\[
\frac{t^{(n)}_{\operatorname{mix}}(\varepsilon)}{t^{(n)}_{\operatorname{mix}}(1 - \varepsilon)}
\geq 
\frac{e^n \log(\frac{1}{2\varepsilon})(1 - \cos(\frac{2\pi}{q}))}{4 \log(nq(1/q^2 - \varepsilon)^{-1})} \to \infty
\]
as $n \to \infty$, which shows that $X_t$ does have cutoff with a fixed even $q$.
Note that $X_t$ does not even exhibit \textit{pre-cutoff}, which requires
\[
\sup_{0 < \varepsilon < 1/2}
    \limsup_{ n\rightarrow \infty} \frac{t^{(n)}_{\operatorname{mix}}(\varepsilon)}{t^{(n)}_{\operatorname{mix}}(1-\varepsilon)} < \infty.
\]

\section{$m\times n$ Contingency Tables}

We now consider the rate-$1$ continuous-time Diaconis-Gangolli random walk on $m \times n$ contingency tables over $\bbZ_q$, whose dynamics were introduced in Section 1.
Again, we assume $\kappa_1 n \leq m \leq \kappa_2 n$ with constants $\kappa_1, \kappa_2 > 0$.
And WLOG we assume $\kappa_1 n$ and $\kappa_2 n$ are integers.

As in the $1 \times n$ case, we will only keep track of the first $m-1$ rows and $n-1$ columns to gain irreducibility and aperiodicity.
Using the analogous notations, we denote the resulting continuous-time process on $\bbZ_q^{(m-1)(n-1)}$ by $\{X_t\}_{t\geq 0}$.
Let $\{Y_s\}_{s \in \bbN}$ be the embedded discrete-time chain of $\{X_t\}_{t\geq 0}$, and assume without loss of generality that $X_0 = Y_0 = 0$.
Again, let $Y \in \bbZ^{(n-1)^2}$ be a lattice-valued random vector with the same law as that of $Y_1$ mapped to $(\bbZ \cap [-q/2, q/2))^{(m-1)(n-1)}$.

Our proof strategy will be similar to sections 4 and 5. We will show that the Diaconis Gangolli walk on $m \times n$ contingency tables over $\bbZ_q$ meets the framework of our generalization in under the asymptotics where $n \rightarrow \infty$, $m \asymp  n$, and $q \gg n^2$. In particular, we need to show the correlation and decay conditions hold. In Section 6.1, we diagonalize and invert the correlation matrix for the walk's natural lift to $\mathbb{Z}^{(m-1) \times (n-1)}$. In section 6.2, we work directly with the characteristic function of the walk's natural lift to show the decay condition. In Section 6.3, we verify that the correlation condition additionally holds by analyzing Schur complements of $\Gamma^{-1}$.

\subsection{Inverting the Correlation Matrix}

In this subsection, we present the inverse $\Gamma^{-1} \in \mathbb{R}^{(m-1)(n-1) \times (m-1)(n-1)}$ of the correlation matrix $\Gamma \in \mathbb{R}^{(m-1)(n-1) \times (m-1)(n-1)}$ of $Y$ for later use.
Note that $\Gamma$ is given by
\begin{align*}
    \Gamma_{(i,j),(i',j')} =\begin{cases}
    1 & (i,j) = (i',j') \\
    -\frac{1}{n-1} & i = i' \text{ and } j \neq j' \\
    - \frac{1}{m-1} & i \neq i' \text{ and } j = j' \\
    \frac{1}{(m-1)(n-1)} & i \neq i' \text{ and } j \neq j'
\end{cases} .
\end{align*}
Also, $Y$ has covariance matrix $\dfrac{4}{mn} \Gamma$. 
We claim that the inverse $\Psi := \Gamma^{-1}$ of $\Gamma$ is given by
\begin{align*}
    \Psi_{(i,j),(i',j')} = \begin{cases}
    \frac{4(m-1)(n-1)}{mn} & (i,j) = (i',j') \\
    \frac{2(m-1)(n-1)}{mn} & i = i' \text{ or } j = j' \ \text{but not both}\\
    \frac{(m-1)(n-1)}{mn} & i \neq i' \text{ and } j \neq j' 
\end{cases} .
\end{align*}
To verify the claim, we will give the eigenvalues and eigenvectors of $\Gamma$ and $\Psi$. 
First, the all-one vector $\mathbbm{1}_{(m-1)(n-1)}$ is an eigenvector of eigenvalue $\dfrac{1}{(m-1)(n-1)}$ for $\Gamma$ and of eigenvalue $(m-1)(n-1)$ for $\Psi$. 
Next, if we define
\begin{align*}
    \delta_{i.}(\ell,k) := \begin{cases} -1 & \ell = 1 \\ 1 & \ell = i \\ 0 & \ell \notin \{1,i \} \end{cases}, \hspace{25pt} \delta_{.j}(\ell,k) := \begin{cases} -1 & k = 1 \\ 1 & k = j \\ 0 & k \notin \{1,j \} \end{cases},
\end{align*}
then for $2 \leq i \leq m-1$, each $\delta_{i.}$ is an eigenvector for $\Gamma$ with eigenvalue $\dfrac{m}{(m-1)(n-1)}$, and eigenvector for $\Psi$ with eigenvalue $\dfrac{(m-1)(n-1)}{m}$. Similarly, one can show that for $2 \leq j \leq n-1$, each $\delta_{.j}$ is an eigenvector for $\Gamma$ with eigenvalue $\dfrac{n}{(m-1)(n-1)}$ and an eigenvector for $\Psi$ with eigenvalue $\dfrac{(m-1)(n-1)}{n}$. 
Lastly, if we define the vectors
\begin{align*}
    \delta_{ij}(\ell,k) := \begin{cases} -1 & \left(\ell = 1 \text{ and } k = 1 \right) \text{ or } \left(\ell=i \text{ and } k = j \right) \\
    1 & \left( \ell = 1 \text{ and } k = j \right) \text{ or } \left( \ell=i \text{ and } k =1 \right) \\
    0 & \ell \notin \{1,i \} \text{ or } k \notin \{1,j \}
    \end{cases}.
\end{align*}
then each $\delta_{ij}$ is an eigenvector for eigenvalue $\dfrac{mn}{(m-1)(n-1)}$ for $\Gamma$, and an eigenvector for eigenvalue $\dfrac{(m-1)(n-1)}{mn}$ for $\Psi$, when $2 \leq i \leq m-1$ and $2 \leq j \leq n-1$. 
Hence, using the spectral decomposition of $\Gamma$ and $\Psi$, we can conclude that $\Psi = \Gamma^{-1}$.

\subsection{Bounding Decay in Characteristic Function for Large $\theta$}

In this subsection, we bound the decay of the characteristic function $\Phi_{Y}(\theta)$ of the $m\times n$ contingency table random walk in the regime $\|\theta\|_\infty > \dfrac{1}{\sqrt{q}}$.
We will eventually plug in $\theta = \dfrac{2\pi y}{q}$, $y \in (\bbZ \cap [-q/2, q/2))^{(m-1)(n-1)}$.
First, note that we can write $\binom{m}{2}\binom{n}{2} \Phi_{Y}(\theta)$ as
\begin{align*}
\sum_{\substack{i>\ell \\ j> k}} \cos(\theta_{ij} - \theta_{ik} - \theta_{\ell j} + \theta_{\ell k}) + \sum_{\substack{i>\ell \\ j}} \cos(\theta_{ij} - \theta_{\ell j}) + \sum_{\substack{j>k \\ i}} \cos(\theta_{ij} - \theta_{ik}) + \sum_{i,j} \cos(\theta_{ij}).
\end{align*}
We may assume that $\theta \in [-\pi, \pi]^{(m-1)(n-1)}$.

\subsubsection{Regime 1: $\dfrac{1}{\sqrt{q}} < \| \theta \|_\infty \leq \dfrac{\pi}{4}$}

Suppose that $\dfrac{1}{\sqrt{q}} < \| \theta \|_\infty \leq \dfrac{\pi}{4}$. 
Recall that $Y$ has covariance matrix $\dfrac{4}{mn}\Gamma$.
Since $\cos(x) \leq 1 - \dfrac{2}{25}x^2$ for $x \in [-\pi, \pi]$, by the second-order Taylor expansion of $\Phi_{Y}$, 
\begin{equation}
\label{decay00}
    \Phi_{Y}(\theta) \leq 1 - \frac{2}{25}\theta^T \left( \frac{4}{mn} \Gamma \right) \theta = 1 - \frac{8}{25 mn} \theta^T \Gamma \theta 
    \leq 
    1 - \frac{8}{25 \kappa_2 n^2} \theta^T \Gamma \theta 
\end{equation}
Recall from Section 6.1 that $\Gamma$ has eigenvalue $\dfrac{1}{(m-1)(n-1)}$ corresponding to the eigenvector $\mathbbm{1}_{(m-1)(n-1)}$, eigenvalue $\dfrac{m}{(m-1)(n-1)}$ with an eigenspace equal to $\operatorname{span}( \{ \delta_{i.} \}_{i=2}^{m-1})$, and eigenvalue $\dfrac{n}{(m-1)(n-1)}$ with an eigenspace $\operatorname{span}(\{ \delta_{.j} \}_{j=2}^{n-1} ))$.
Additionally, $\Gamma$ has an eigenvalue $\dfrac{mn}{(m-1)(n-1)}$ with an eigenspace equal to 
\[
\operatorname{span}( \{ \delta_{ij} \} | i \in \{2, \ldots, m-1\},j \in \{2,\ldots, n-1 \}).
\]
We write
\begin{align*}
    \overline{\theta} :=  \frac{\sum_{i,j} \theta_{ij}}{(m-1)(n-1)} , \hspace{15pt} \overline{\theta}_{i.} := \frac{ \sum_{j=1}^{n-1} \theta_{ij}}{n-1} , \hspace{15pt} \overline{\theta}_{.j} :=  \frac{ \sum_{i=1}^{m-1} \theta_{ij}}{m-1} . 
\end{align*} 
If $|\overline{\theta} | \geq \dfrac{1}{8 \sqrt{q}}$, then
\begin{equation}
\label{decay01}
    \theta^T \Gamma \theta \geq \frac{1}{(m-1)(n-1)} \| \overline{\theta} \mathbbm{1}_{(m-1)(n-1)} \|_2^2 = \frac{1}{64q}. 
\end{equation}
So we assume that $| \overline{\theta} | < \dfrac{1}{8 \sqrt{q}}$. 
Also, suppose that there exists either a column or row (we will first assume a row $i$) such that $| \overline{\theta}_{i.} | \geq \dfrac{1}{4\sqrt{q}}$. 
Then there must be some additional row $i'$ such that
\begin{equation}
\label{decay02}
    |\overline{\theta}_{i.} - \overline{\theta}_{i'.} | \geq |\overline{\theta}_{i.} - \overline{\theta} | \geq \frac{1}{8 \sqrt{q}}.
\end{equation}
Since $w := \delta_{i.} - \delta_{i'.}$ is an eigenvector for $\Gamma$ with eigenvalue $\dfrac{m}{(m-1)(n-1)}$, we know that the projection of $\theta$ onto the span of $w$ is equal to
\begin{align*}
    \frac{\overline{\theta}_{i.} - \overline{\theta}_{i'.}}{2} w
\end{align*}
with $\| w \|_2^2 = 2(n-1)$. 
It follows from (\ref{decay02}) that
\begin{equation}
\label{decay03}
\begin{aligned}
    \theta^T \Gamma \theta &\geq \left( \frac{\overline{\theta}_{i.} - \overline{\theta}_{i'.}}{2} \right)^2 w^T \Gamma w \\ 
    &= \left( \frac{\overline{\theta}_{i.} - \overline{\theta}_{i'.}}{2} \right)^2 \frac{m}{(m-1)(n-1)} \| w \|_2^2  \\
    & \geq \frac{1}{4} \left(\frac{1}{8 \sqrt{q}} \right)^2 \frac{2 m}{m-1} \\
    &\geq \frac{1}{128  q}.
\end{aligned}
\end{equation}
Now if there is a column $j$ such that $|\overline{\theta}_{. j}| \geq \dfrac{1}{4\sqrt{q}}$, then a similar argument gives
\begin{equation}
    \label{decay03'}
     \theta^T \Gamma \theta \geq \frac{1}{128 q}.
\end{equation}

It remains to consider the case where
\begin{align*}
    |\overline{\theta}| < \frac{1}{8\sqrt{q}}, \hspace{15pt} |\overline{\theta}_{i.}| < \frac{1}{4 \sqrt{q}} \text{ for } i \in \{1,\ldots, m \}, \quad \text{and} \quad |\overline{\theta}_{.j}| < \frac{1}{4 \sqrt{q}} \text{ for } j \in \{1,\ldots, n \}.
\end{align*}
Let $\hat{\theta}$ be the projection of $\theta$ onto $\operatorname{span} \left( \mathbbm{1}_{(m-1)(n-1)}, \{ \delta_{i.}\}_{i=1}^\infty, \{\delta_{.j} \}_{j=1}^\infty \right)$. 
It follows that
\begin{equation}
   \label{decay04} 
    \hat{\theta}(\ell,k) = \overline{\theta}_{\ell.} + \overline{\theta}_{.k} - \overline{\theta}.
\end{equation}
Notice that
\begin{equation} 
\label{decay05}
    \theta^T \Gamma \theta \geq \left(\theta - \hat{\theta} \right)^T \Gamma \left(\theta - \hat{\theta} \right) \geq \rVert \theta - \hat{\theta} \rVert^2_2.
\end{equation}
By assumptions and (\ref{decay04}), we have
\begin{equation}
    \label{decay06}
    \| \theta - \hat{\theta} \|_\infty \geq \| \theta \|_\infty - \frac{1}{4 \sqrt{q}} - \frac{1}{4 \sqrt{q}} - \frac{1}{8 \sqrt{q}} \geq  \frac{3}{8 \sqrt{q}}.
\end{equation}
Putting (\ref{decay05}) and (\ref{decay06}) together yields
\begin{equation}
    \label{decay07}
    \theta^T \Gamma \theta 
    \geq 
    \| \theta - \hat{\theta} \|_2^2 
    \geq 
    \| \theta - \hat{\theta} \|_\infty^2
     \geq
     \frac{9}{64 q}.
\end{equation}

Now with (\ref{decay00}), (\ref{decay01}), (\ref{decay03}), (\ref{decay03'}), and (\ref{decay07}), we conclude that when $\dfrac{1}{\sqrt{q}} < \| \theta \|_\infty \leq \dfrac{\pi}{4}$,
\begin{equation}
\label{decay10}
    \Phi_{Y}(\theta) \leq 1 - \frac{8}{25 \kappa_2 n^2} \frac{1}{128 q} = 1 - \frac{1}{400 \kappa_2 n^2 q}.
\end{equation}

\subsubsection{Regime 2: $\dfrac{\pi}{4} \leq \|\theta\|_\infty \leq \pi$}

We will now bound the characteristic function $\Phi_{Y}(\theta)$ for the regime where $\| \theta \|_\infty \in [\frac{\pi}{4}, \pi ]$. 
Recall that we can write $\binom{m}{2} \binom{n}{2} \Phi_{Y}(\theta)$ as
\begin{equation}
\label{tablecharfunction}
      \sum_{\substack{i>\ell \\ j> k}} \cos(\theta_{ij} - \theta_{ik} - \theta_{\ell j} + \theta_{\ell k}) + \sum_{\substack{i>\ell \\ j}} \cos(\theta_{ij} - \theta_{\ell j}) + \sum_{\substack{j>k \\ i}} \cos(\theta_{ij} - \theta_{ik}) + \sum_{i,j} \cos(\theta_{ij})
\end{equation}
Let $[n] := \{1, \dots, n\}$ and
\begin{align*}
    \mathcal{A} := \{ (i,j) \in [m - 1] \times [n - 1] : |\theta_{ij} | > 1/8 \}.
\end{align*}
Suppose that $|\cA| \geq \dfrac{(m-1)(n-1)}{32}$. Then we can write
\begin{equation}
\label{decay11}
\begin{aligned}
    \Phi_{Y}(\theta) &\leq 1 - \dfrac{ \sum_{(i,j) \in \mathcal{A}} 1- \cos(\theta_{ij})}{ {\binom{m}{2} \binom{n}{2}}} \\
    & \leq 1 - \dfrac{(m-1)(n-1) (1 - \cos(1/8))}{32  {\binom{m}{2}\binom{n}{2}}} \\
    & = 1 - \frac{1 - \cos(1/8)}{8mn},
\end{aligned}
\end{equation}
where we replaced all the elements of the characteristic function aside from the $\cos(\theta_{ij})$ terms with $(i,j) \in \mathcal{A}$ with $1$. 

Next, suppose that $|\cA| \leq \dfrac{(m-1)(n-1)}{32}$.
Pick $(i^*, j^*)$ such that $|\theta_{i^* j^*}| = \|\theta\|_\infty$.
Define
\[
\cB := \{i \in [m-1] : |\theta_{i j^*}| > 1/4\}
\]
and
\[
\cC := \{j \in [n-1] : |\theta_{i^* j}| > 1/4\}.
\]
Suppose $|\cB| \leq \dfrac{m-1}{2}$ and $|\cC| \leq \dfrac{n-1}{2}$.
We look at the terms in (\ref{tablecharfunction}) of the form $\cos(\theta_{ij} - \theta_{ik} - \theta_{\ell j} + \theta_{\ell k})$.
We would replace all the terms by $1$ except those of the form $\cos(\theta_{i^* j^*} - \theta_{i^* j} - \theta_{i j^*} + \theta_{ij})$.
There are at least $\dfrac{(m-1)(n-1)}{4}$ terms of this form with $i \in \cB^c$ and $j \in \cC^c$.
Since $|\cA| \leq \dfrac{(m-1)(n-1)}{32}$, there are at least $\dfrac{7(m-1)(n-1)}{32}$ of terms $\cos(\theta_{i^* j^*} - \theta_{i^* j} - \theta_{i j^*} + \theta_{ij})$ such that $i \in \cB^c, j\in \cC^c$, and $|\theta_{ij}| \leq 1/8$.
Therefore,
\begin{equation}
    \label{decay12}
    \begin{aligned}
        \Phi_{Y}(\theta) 
        &\leq 1 - \frac{\sum 1 - \cos(\theta_{i^* j^*} - \theta_{i^* j} - \theta_{i j^*} + \theta_{ij})}{\binom{m}{2}\binom{n}{2}} \\
        &\leq 1 - \frac{\frac{7}{32} (m-1)(n-1) (1 - \cos(\pi/4 - 5/8))}{\binom{m}{2}\binom{n}{2}} \\
        &= 1 - \frac{7(1 - \cos(\pi/4 - 5/8))}{8mn}
    \end{aligned}
\end{equation}
where on the first line, the summation ranges over all the pairs $(i,j)$ such that $i \in \cB^c, j\in \cC^c$, and $|\theta_{ij}| \leq 1/8$.
Now, we suppose that $|\cA| \leq \dfrac{(m-1)(n-1)}{32}$ and $|\cC| > \dfrac{n-1}{2}$.
Note that the case of $|\cA| \leq \dfrac{(m-1)(n-1)}{32}$ and $|\cB| > \dfrac{m-1}{2}$ would follow in an analogous way.
We look at the terms in (\ref{tablecharfunction}) of the form $\cos(\theta_{i^* j} - \theta_{\ell j})$ where $j \in \cC$ and $\ell \neq i^*$.
There are at least $\dfrac{(n-1)(m-2)}{2} \geq \dfrac{(m-1)(n-1)}{4}$ pairs of $(\ell,j)$ such that $j \in \cC$ and $\ell \neq i^*$.
So, there are at least $\dfrac{7(m-1)(n-1)}{32}$ pairs of $(\ell,j)$ such that $j \in \cC$, $\ell \neq i^*$, and $|\theta_{\ell j}| \leq 1/8$.
Therefore,
\begin{equation}
    \label{decay13}
    \begin{aligned}
        \Phi_{Y}(\theta)
        &\leq 1 - \frac{\sum 1 - \cos(\theta_{i^* j} - \theta_{\ell j})}{\binom{m}{2}\binom{n}{2}} \\
        &\leq 1 - \frac{\frac{7}{32}(m-1)(n-1)(1-\cos(1/8))}{\binom{m}{2}\binom{n}{2}} \\
        &= 1 - \frac{7(1 - \cos(1/8))}{8mn}
    \end{aligned}
\end{equation}
where the summation ranges over all the pairs $(\ell,j)$ such that $j \in \cC$, $\ell \neq i^*$, and $|\theta_{\ell j}| \leq 1/8$.

\subsubsection{Combining the Two Regimes of Large $\theta$}

Now since $m \leq \kappa_2 n$, combining (\ref{decay10}), (\ref{decay11}), (\ref{decay12}), and (\ref{decay13}), we conclude that for all large $n$ and $q$,
\begin{equation}
    \label{decay14}
    \Phi_{Y}(\theta) \leq 1 - \frac{1}{400 \kappa_2 n^2 q}
\end{equation}
when $\theta \in [-\pi, \pi]^{(m-1)(n-1)}$ satisfies $\|\theta\|_\infty > \dfrac{1}{\sqrt{q}}$.
So for $y \in (\bbZ \cap [-q/2, q/2))^{(m-1)(n-1)}$ with $\|y\|_\infty > \dfrac{\sqrt{q}}{2\pi}$,
\begin{equation}
\label{decay15}
    \exp\left(2t\left(\Phi_{Y}\left(\frac{2\pi y}{q}\right) - 1\right)\right) \leq \exp\left(-\frac{t}{200 \kappa_2 n^2 q}\right).
\end{equation}
By (\ref{generalmixingtime}) with $r = 4$, $\sigma^2 = \dfrac{4(m-1)(n-1)}{mn}$, the mixing time is 
\begin{equation}
\label{n*nmixingtime}
  t = \dfrac{m n q^2 \log(\alpha (m-1)(n-1))}{16 \pi^2}\left(1 - \dfrac{4}{3q}\right)^{-1},
\end{equation}
where $\alpha \geq 1$ is independent of $n$.
At the mixing time, since $\alpha \geq 1$ and $m\geq \kappa_1 n$, we can write $(\ref{decay15})$ as
\begin{equation}
    \label{decay16}
    \exp\left(2t\left(\Phi_{Y}\left(\frac{2\pi y}{q}\right) - 1\right)\right) \leq \left((m-1)(n-1)\right)^{-\dfrac{\kappa_1 q}{3200 \kappa_2 \pi^2}}
\end{equation}
when $\|y\|_\infty \geq \dfrac{\sqrt{q}}{2\pi}$.
Similar to (\ref{elementcount}), the number of $y \in (\bbZ \cap [-q/2, q/2))^{(m-1)(n-1)}$ with $\|y\|_\infty > \dfrac{\sqrt{q}}{2\pi}$ is $\cO(q^{(m-1)(n-1)})$.
Thus if $q \gg m n$, then at the mixing time (\ref{n*nmixingtime}), we conclude that as $n \to \infty$,
\begin{equation}
    \label{decay17}
    \sum_{\|y\|_\infty > \frac{\sqrt{q}}{2\pi}}\exp\left(2t\left(\Phi_{Y}\left(\frac{2\pi y}{q}\right) - 1 \right)\right) \to 0.
\end{equation}

\subsection{Bounding Decay in Characteristic Function for Small $\theta$}

In this subsection, we bound the decay of the characteristic function $\Phi_{Y}(\theta)$ in the regime $\|\theta\|_\infty \leq \dfrac{1}{\sqrt{q}}$. 
\subsubsection{The Schur Complement of a Submatrix of $\Psi$}
For convenience, we write $A := \dfrac{mn}{(m-1)(n-1)} \Psi$ and thus
\begin{align*}
    A_{(i,j),(i',j')} = \begin{cases}
    4 & (i,j) = (i',j') \\
    2 & i = i' \text{ or } j = j' \ \text{but not both} \\
    1 & i \neq i' \text{ and } j \neq j' 
\end{cases}.
\end{align*}
Fix any $\ell \in [m-1], k \in [n-1]$.
Let $A^* = A^*(\ell,k)$ be the principal submatrix of $A$ that only contains indices from the first $\ell$ rows and $k$ columns of the contingency table. 
We further partition $A^*$ as
\begin{align*}
    A^* = \begin{bmatrix} A^*_{11} & A^*_{12} \\ A^*_{21} & A^*_{22} \end{bmatrix}
\end{align*}
where $A'_{11}$ is the principal submatrix containing all indices except $(\ell,k)$, and $A'_{22} = 4$ corresponds to index $(\ell,k)$. 
We have the following lemma.
\begin{lemma}
\label{shurcomplement}
Let $A^* = A^*(\ell, k)$ be as above.
    Then the Schur complement is    
    \begin{align*}
        A^*_{22} - A^*_{21} (A^*_{11})^{-1} A^*_{12} = 1 + \frac{\ell + k + 1}{\ell k}.
    \end{align*}
\begin{proof}
    Define the vector $v \in \bbR^{\ell k - 1}$ by
    \begin{align*}
    v_{(i,j)} := \begin{cases} 1/k & i = \ell \text{ and } j \neq k\\
    1/\ell & i \neq \ell \text{ and } j = k \\
    -\dfrac{1}{\ell k} & i \neq \ell \text{ and } j \neq k 
    \end{cases} \end{align*}
for $1 \leq i \leq \ell$, $1 \leq j \leq k$, $(i,j) \neq (\ell,k)$. 
We will show that $A^*_{11} v = A^*_{12}$, and consequently, $(A^*_{11})^{-1} A^*_{12} = v$.

We break the calculation into three cases. 
First, we consider an index $(i,k)$, $i \neq \ell$. 
As indicated in Figure \ref{figure3}, within the first $\ell$ rows and $k$ columns of the contingency table, there exist $\ell-2$ indices in the $k$th column and a row not the $i$th or the $\ell$th, $k-1$ indices in the $i$th row and a column not the $k$th, $k-1$ indices in the $\ell$th row and a column not the $k$th, and $(k-1)(\ell-2)$ indices that are neither in the $i$th or the $\ell$th row, nor in the $k$th column. 
A direct calculation then yields
\begin{align*}
    (A^*_{11} v)_{(i,k)} &= \frac{4}{\ell} + \frac{2(\ell-2)}{\ell} - \frac{2(k-1)}{\ell k} + \frac{k-1}{k} - \frac{(k-1)(\ell-2)}{\ell k}
    = 2.
\end{align*}

\vspace{5pt}
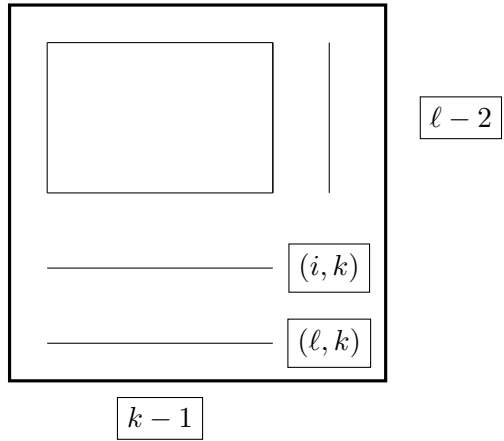
\begin{figure}[H]
\centering
\begin{tikzpicture}
\draw[black, very thick] (0,0) rectangle (5,5);
\node[draw] at (4.25,0.5) {$(\ell, k)$};
\node[draw] at (4.25,1.5) {$(i,k)$};
\draw (4.25,2.5) -- (4.25,4.5);
\draw (0.5,0.5) -- (3.5,0.5);
\draw (0.5,1.5) -- (3.5,1.5);
\draw (0.5,2.5) -- (3.5,2.5);
\draw (3.5,2.5) -- (3.5,4.5);
\draw (3.5,2.5) -- (3.5,4.5);
\draw (0.5,4.5) -- (3.5,4.5);
\draw (0.5,4.5) -- (0.5,2.5);
\node[draw] at (2,-0.5) {$k-1$};
\node[draw] at (6,3.5) {$\ell-2$};
\end{tikzpicture}
\caption{
Diagram for an index $(i,k)$ with a shared $k$th column with $(\ell, k)$ in the submatrix containing the first $\ell$ rows and $k$ columns of the $m \times n$ contingency table.  
Note that this diagram is only for counting purposes, and the rows and columns may have been permuted.}
\label{figure3}
\end{figure}

Next, we deal with the case with an index $(\ell,j)$ with $j \neq k$. 
Analogously, in the first $\ell$ rows and $k$ columns of the contingency table, there are $k-2$ indices in the $\ell$th row and a column not the $j$th or the $k$th, $\ell-1$ indices in the $k$th column and a row not the $\ell$th, $\ell-1$ indices in the $j$th column and a row not the $\ell$th, and $(\ell-1)(k-2)$ indices neither in the $\ell$th row nor in the $k$th or the $j$th column. 
Therefore,
\[
    (A^*_{11} v)_{(\ell,j)} = \frac{4}{k} + \frac{2(k-2)}{k} - \frac{2(\ell-1)}{\ell k} + \frac{\ell-1}{\ell} - \frac{(k-2)(\ell-1)}{\ell k} 
    = 2.
\]
In the last case, we check for an index $(i,j)$ with $i \neq \ell$, $j \neq k$. 
As in Figure \ref{figure4}, in the first $\ell$ rows and $k$ columns of the contingency table, there are $\ell-2$ indices in the $k$th column and a row not the $\ell$th or the $i$th, $k-2$ indices in the $\ell$th row and a column not the $k$th or the $j$th, $\ell-2$ indices in the $j$th column and a row not the $\ell$th or the $i$th, $k-2$ indices in the $i$th row and a column not the $k$th or the $j$th, and $(\ell-2)(k-2)$ indices that are neither in the $\ell$th or $i$th row nor in the $k$th or the $j$th column. This then gives
\begin{align*}
    (A^*_{11} v)_{(i,j)} 
    &= -\frac{4}{\ell k} + \frac{2}{\ell} + \frac{2}{k} + \frac{\ell-2}{\ell} + \frac{k-2}{k} - \frac{2(\ell-2)}{\ell k} - \frac{2(k-2)}{\ell k} - \frac{(\ell-2)(k-2)}{\ell k} \\
    &= 1.
\end{align*}

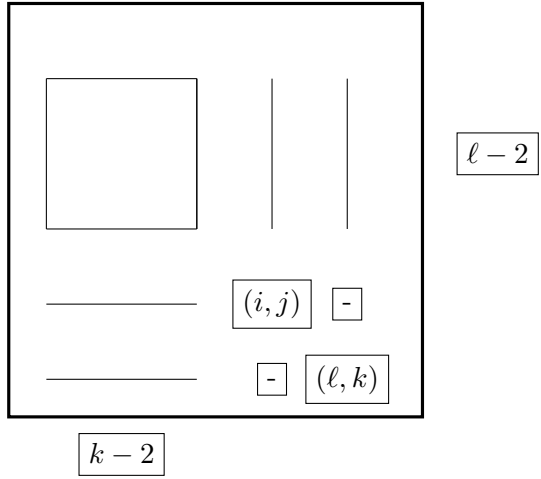
\begin{figure}[H]
\centering
\begin{tikzpicture}
\draw[black, very thick] (0,0) rectangle (5.5,5.5);
\node[draw] at (4.5,0.5) {$(\ell, k)$};
\node[draw] at (3.5,1.5) {$(i,j)$};
\draw (4.5,2.5) -- (4.5,4.5);
\draw (3.5,2.5) -- (3.5,4.5);
\draw (0.5,0.5) -- (2.5,0.5);
\draw (0.5,1.5) -- (2.5,1.5);
\draw (0.5,2.5) -- (2.5,2.5);
\draw (2.5,2.5) -- (2.5,4.5);
\draw (2.5,2.5) -- (2.5,4.5);
\draw (0.5,4.5) -- (2.5,4.5);
\draw (0.5,4.5) -- (0.5,2.5);
\node[draw] at (1.5,-0.5) {$k-2$};
\node[draw] at (6.5,3.5) {$\ell-2$};
\node[draw] at (4.5,1.5) {-};
\node[draw] at (3.5,0.5) {-};
\end{tikzpicture}
\caption{Diagram for an index $(i,j)$ that does not share a row or a column with $(\ell, k)$ in the submatrix containing the first $\ell$ rows and $k$ columns of the $m \times n$ contingency table.  
Note that this diagram is only for counting purposes, and the rows and columns may have been permuted.}
\label{figure4}
\end{figure}

Therefore, we have $A^*_{11}v = A^*_{12}$ and thus $(A^*_{11})^{-1} A^*_{12} = v$. 
Now, since there are $k-1$ indices with row $\ell$ other than $(\ell,k)$, and $\ell-1$ indices with column $k$ other than $(\ell,k)$, a direct computation yields
\begin{align*}
    A^*_{21} v = \frac{2(\ell-1)}{\ell} + \frac{2(k-1)}{k} - \frac{(k-1)(\ell-1)}{\ell k} 
    = 3 + \frac{-\ell -k - 1}{\ell k}.
\end{align*}
Hence,
\begin{align*}
    A^*_{22} - A^*_{21} (A^*_{11})^{-1} A^*_{12} 
    = A^*_{22} - A^*_{21} v
    = 4 - \left(3 + \frac{-\ell-k-1}{\ell k}\right)
    = 1 + \frac{\ell + k + 1}{\ell k}.
\end{align*}
\end{proof}
\end{lemma}

\subsubsection{The Correlation Condition}
We now show that the correlation matrices $\Gamma_{n}$ for the random walk on the $m \times n$ contingency table meet the Correlation Condition.

\begin{lemma}
\label{n*ncorrelationcondition}
    Let $(\Gamma_n)_{n=1}^\infty$ be the sequence of correlation matrices corresponding to the contingency table walk on an $n \times n$ table. 
    Then $\Gamma_n$ meets the Correlation Condition with the bounding function
    \begin{align*}
        g(\alpha) = \frac{\kappa_2^2(1 + 2e^{4/e} + 2e^{8/e} + e^{32/e})}{\kappa_1 \alpha^{1/4}}.
    \end{align*}
    \begin{proof}
        Recall from Section 6.1 that $\Gamma_{n}^{-1} = \Psi_n$ with
\begin{align*}
    \Psi_n((i,j),(i',j')) = \begin{cases}
    \frac{4(m-1)(n-1)}{mn} & (i,j) = (i',j') \\
    \frac{2(m-1)(n-1)}{mn} & i = i' \text{ or } j = j' \ \text{but not both} \\
    \frac{(m-1)(n-1)}{mn} & i \neq i' \text{ and } j \neq j' 
\end{cases}.
\end{align*}
It follows that $\sup_n \| \Psi_n \|_\infty = 4$ which confirms condition (i).
In order to check condition (ii), we first specify an ordering of the elements of $\Psi_n$ in order to compute each respective Schur complement:
the first $n-1$ elements will consist of the first row of the table (incrementing by column index), second $n-1$ elements will consist of row two in the same order, continuing for each subsequent row. 

Now, suppose that we are at an index $(\ell,k)$ of the contingency table such that $\ell > 1$, $k > 1$. 
The correlation matrix $\Psi_{n,(\ell,k)}$ corresponds to all the first $\ell-1$ rows and the first $k$ entries of row $\ell$ in the contingency table. 
We call the intersection between the first $\ell$ rows and $k$ columns in the contingency table the \textit{interior region}. 
See Figure \ref{figure5} below.

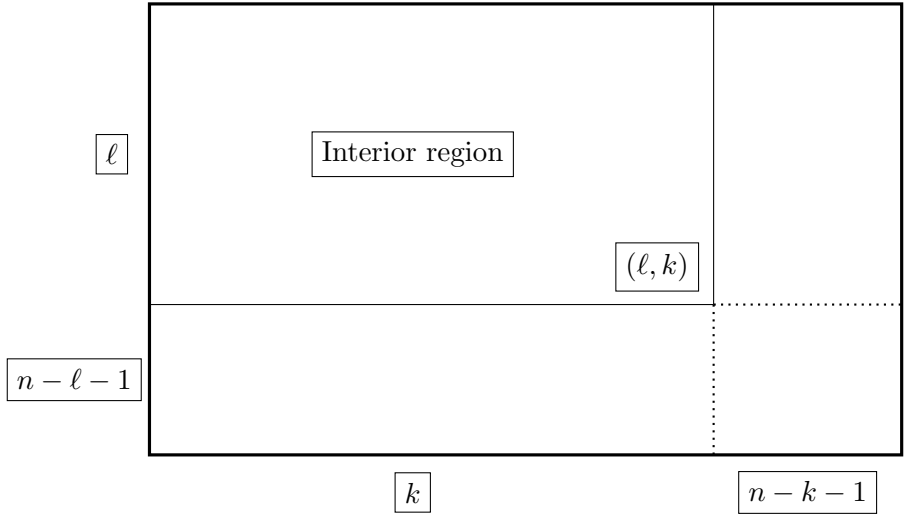
\begin{figure}[H]
\centering
\begin{tikzpicture}
\draw[black, very thick] (-2,-2) rectangle (8,4);
\draw (5.5,4) -- (5.5,0);
\draw (-2,0) -- (5.5,0);
\draw[dotted, thick] (5.5, 0) -- (5.5, -2);
\draw[dotted, thick] (5.5, 0) -- (8, 0);
\node[draw] at (4.75,0.5) {$(\ell, k)$};
\node[draw] at (1.5,2) {Interior region};
\node[draw] at (1.5,-2.5) {$k$};
\node[draw] at (-2.5,2) {$\ell$};
\node[draw] at (-3,-1) {$m-\ell-1$};
\node[draw] at (6.75,-2.5) {$n-k-1$};
\end{tikzpicture}
\caption{Diagram for the \textit{interior region} formed by all the entries in first $\ell$ rows and $k$ columns in the contingency table.}
\label{figure5}
\end{figure}

Let $M_{n,(\ell,k)}$ be the principal submatrix of $\Psi_{n,(\ell,k)}$ which contains all the corresponding indices in the interior region, and partition \begin{align*}
    M_{n,(\ell,k)} = \begin{bmatrix} M_{n,(\ell,k)}^{(11)} & M_{n,(\ell,k)}^{(12)} \\ M_{n,(\ell,k)}^{(21)} & M_{n,(\ell,k)}^{(22)} \end{bmatrix}
\end{align*}
so that $M_{n,(\ell,k)}^{(22)} \in \bbR$ corresponds to the element $(\ell,k)$. 
Then, since $M_{n,(\ell,k)}$ is a principal submatrix of $\Psi_{n,(\ell,k)}$, by a similar decomposition as (\ref{quadratic1}),
\begin{align*}
    M_{n,(\ell,k)}^{(22)} -  M_{n,(\ell,k)}^{(21)} \left( M_{n,(\ell,k)}^{(11)}\right)^{-1}  M_{n,(\ell,k)}^{(12)} \geq  \Psi_{n,(\ell,k)}^{(22)} -  \Psi_{n,(\ell,k)}^{(21)} \left( \Psi_{n,(\ell,k)}^{(11)}\right)^{-1}  \Psi_{n,(\ell,k)}^{(12)}.
\end{align*}
And by Lemma \ref{shurcomplement}, we know
\begin{align*}
    M_{n,(\ell,k)}^{(22)} -  M_{n,(\ell,k)}^{(21)} \left( M_{n,(\ell,k)}^{(11)}\right)^{-1}  M_{n,(\ell,k)}^{(12)} = \left(\frac{(m-1)(n-1)}{mn}\right)\left(1 + \frac{\ell + k + 1}{\ell k}\right).
\end{align*}
Thus
\begin{align*}
    a_{\ell,k}^{(n)} :=
    \left( \Psi_{n,(\ell,k)}^{(22)} -  \Psi_{n,(\ell,k)}^{(21)} \left( \Psi_{n,(\ell,k)}^{(11)}\right)^{-1}  \Psi_{n,(\ell,k)}^{(12)} \right)^{-1} \geq \frac{\ell k}{(\ell+1)(k+1)} .
\end{align*}

Now, if $\ell=1$ or $k=1$, say $\ell = 1$, then note that $\Psi_{n,(1,k)} = \dfrac{2(m-1)(n-1)}{mn}(I_k + J_k)$.
So by (\ref{quadratic5}),
\begin{align*}
    a_{\ell,k}^{(n)} =
    \left( \Psi_{n,(\ell,k)}^{(22)} -  \Psi_{n,(\ell,k)}^{(21)} \left( \Psi_{n,(\ell,k)}^{(11)}\right)^{-1}  \Psi_{n,(\ell,k)}^{(12)} \right)^{-1} \geq \frac{\ell k}{(\ell+1)(k+1)}
\end{align*} 
when $\ell = 1$ or $k = 1$.
Using these results, we can write
\begin{align*}
    \sum_{\ell=1}^{m-1} \sum_{k=1}^{n-1} \left( \frac{1}{\alpha (m-1)(n-1)} \right)^{a_{\ell,k}^{(n)}} &\leq \sum_{\ell=1}^{m-1} \sum_{k=1}^{n-1} \left( \frac{1}{\alpha (m-1)(n-1)} \right)^{\frac{\ell k}{(\ell+1)(k+1)}} \\
    & \leq \frac{1}{\alpha^{1/4}} \sum_{\ell=1}^{m-1} \sum_{k=1}^{n-1} \left( \frac{1}{mn} \right)^{\frac{\ell k}{(\ell+1)(k+1)}} \\
    & \leq \frac{\kappa_2^2(1 + 2e^{4/e} + 2e^{8/e} + e^{32/e})}{\kappa_1 \alpha^{1/4}}
\end{align*}
where the last inequality follows by Lemma \ref{doublesummation}.
    \end{proof}
\end{lemma}

Now, recall from (\ref{n*nmixingtime}) that the mixing time is 
\[
  t = \dfrac{m n q^2 \log(\alpha (m-1)(n-1))}{16 \pi^2}\left(1 - \dfrac{4}{3q}\right)^{-1}.
\]
At the mixing time, by (\ref{eigenvaluestoquadraticform}), Lemma \ref{upperboundsmally}, and Lemma \ref{n*ncorrelationcondition},
\begin{equation}
\label{n*nsmall}
\sum_{y \in \cM} \exp\left(2t\left(\Phi_{Y}\left(\frac{2 \pi y}{q}\right) - 1\right)\right) \leq \exp\left(\frac{\kappa_2^2(6 + 12e^{4/e} + 12e^{8/e} + 6e^{32/e})}{\kappa_1 \alpha^{1/4}}\right).
\end{equation}

\subsection{Mixing Time Upper Bound}

We will conclude the mixing time upper bound for the random walks on $m\times n$ contingency tables.
For any $\varepsilon \in (0,1)$, set 
\[
\alpha := \inf\left\{z \geq 1 : 4\varepsilon^2 \geq \exp\left(\frac{\kappa_2^2(6 + 12e^{4/e} + 12e^{8/e} + 6e^{32/e})}{\kappa_1 z^{1/4}}\right) - 1\right\}.
\]
By the $\ell^2$-bound, (\ref{decay17}), and (\ref{n*nsmall}), when $t = \dfrac{m n q^2 \log(\alpha (m-1)(n-1))}{16 \pi^2}\left(1 - \dfrac{4}{3q}\right)^{-1}$,
\[
d(t) \leq \varepsilon + o(1)
\]
as $n \to \infty$.

\subsection{Mixing Time Lower Bound}

We will now compute the mixing time lower bound for the $m \times n$ contingency table walk in continuous time. We showed that the discrete-time increment $Y$ has equivariant coordinates, with each coordinate having a marginal variance of $\dfrac{4}{mn}$. 
We also note that the off-diagonal entries of the correlation matrix take values of $\dfrac{-1}{m-1}, \dfrac{-1}{n-1}$, or $\dfrac{1}{(m-1)(n-1)}$. 
Since at each jump the $m \times n$ contingency table random walk updates at most four coordinates, each by $\pm 1$, we have that $\|Y\|_1 \leq 4$. 
Therefore, applying (\ref{generallowerbound}) using that $\kappa_1 n \leq m \leq \kappa_2 n$, and setting $r = 4$, $\sigma^2 = \dfrac{4(m-1)(n-1)}{mn}$, we obtain
\[
    t_{\operatorname{mix}}(\varepsilon) \geq \frac{mn q^2}{16\pi^2}\left(\log\left( (m-1)(n-1) \right) + \log\left(\frac{\varepsilon^{-1}-1}{5}\right)\right)
\]
for all $n$ large enough, assuming $q \gg \sqrt{\log n}$.

Now combining Sections 6.4 and 6.5 completes the proof of Theorem \ref{main3}.

\section{Acknowledgments}

The authors thank David Sivakoff for valuable discussions, feedback, and review of the manuscript.

ZF is supported by the Department of Mathematics of The Ohio State University, and AH is supported by the Department of Statistics of The Ohio State University, both via graduate teaching associateship.

\bibliographystyle{amsplain}
\bibliography{refs}

\providecommand{\bysame}{\leavevmode\hbox to3em{\hrulefill}\thinspace}
\providecommand{\MR}{\relax\ifhmode\unskip\space\fi MR }
\providecommand{\MRhref}[2]{%
  \href{http://www.ams.org/mathscinet-getitem?mr=#1}{#2}
}
\providecommand{\href}[2]{#2}
\begin{thebibliography}{10}

\bibitem{AldousDiaconis}
David~J. Aldous and Persi Diaconis, \emph{Shuffling cards and stopping-times}, American Mathematical Monthly \textbf{93} (1986), 333--348.

\bibitem{BayerDiaconis}
Dave Bayer and Persi Diaconis, \emph{{Trailing the Dovetail Shuffle to its Lair}}, The Annals of Applied Probability \textbf{2} (1992), no.~2, 294 -- 313.

\bibitem{Benedicks}
Michael Benedicks, \emph{An estimate of the modulus of the characteristic function of a lattice distribution with application to remainder term estimates in local limit theorems}, The Annals of Probability \textbf{3} (1975), no.~1, 162–165.

\bibitem{ChungGrahamYau}
F.~R.~K. Chung, R.~L. Graham, and S.-T. Yau, \emph{On sampling with markov chains}, Random Structures \& Algorithms \textbf{9} (1996), no.~1-2, 55--77.

\bibitem{CryanDyerGoldbergJerrumMartin}
M.~Cryan, M.~Dyer, L.A. Goldberg, M.~Jerrum, and R.~Martin, \emph{Rapidly mixing markov chains for sampling contingency tables with a constant number of rows}, The 43rd Annual IEEE Symposium on Foundations of Computer Science, 2002. Proceedings., 2002, pp.~711--720.

\bibitem{Diaconis}
Persi Diaconis, \emph{Group representations in probability and statistics}, Institute of Mathematical Statistics Lecture Notes---Monograph Series, 11, Institute of Mathematical Statistics, Hayward, CA, 1988. \MR{MR964069 (90a:60001)}

\bibitem{DiaconisGangolli}
Persi Diaconis and Anil Gangolli, \emph{Rectangular arrays with fixed margins}, Discrete Probability and Algorithms (New York, NY) (David Aldous, Persi Diaconis, Joel Spencer, and J.~Michael Steele, eds.), Springer New York, 1995, pp.~15--41.

\bibitem{DiaconisSaloff-Coste}
Persi Diaconis and Laurent Saloff-Coste, \emph{{Comparison Theorems for Reversible Markov Chains}}, The Annals of Applied Probability \textbf{3} (1993), no.~3, 696 -- 730.

\bibitem{DiaconisSaloff-Coste2006}
Persi Diaconis and Laurent Saloff-Coste, \emph{Separation cut-offs for birth and death chains}, The Annals of Applied Probability \textbf{16} (2007).

\bibitem{DiaconisShahshahani}
Persi Diaconis and Mehrdad Shahshahani, \emph{Generating a random permutation with random transpositions}, Zeitschrift f{\"u}r Wahrscheinlichkeitstheorie und Verwandte Gebiete \textbf{57} (1981), 159--179.

\bibitem{DiaconisSturmfels}
Persi Diaconis and Bernd Sturmfels, \emph{{Algebraic algorithms for sampling from conditional distributions}}, The Annals of Statistics \textbf{26} (1998), no.~1, 363 -- 397.

\bibitem{DingLubetzkyPeres}
Jian Ding, Eyal Lubetzky, and Yuval Peres, \emph{Total variation cutoff in birth-and-death chains}, Probability Theory and Related Fields \textbf{146} (2010), 61--85.

\bibitem{DyerGreenhill}
Martin Dyer and Catherine Greenhill, \emph{Polynomial-time counting and sampling of two-rowed contingency tables}, Theoretical Computer Science \textbf{246} (2000), no.~1, 265--278.

\bibitem{GangulyLubetzkyMartinelli}
Shirshendu Ganguly, Eyal Lubetzky, and Fabio Martinelli, \emph{Cutoff for the east process}, Communications in Mathematical Physics \textbf{335} (2013), 1287--1322.

\bibitem{hermon2018supplementary}
Jonathan Hermon and Sam Olesker-Taylor, \emph{Supplementary material for random cayley graphs project}, arXiv preprint arXiv:1810.05130 (2018).

\bibitem{Hernek}
Diane Hernek, \emph{Random generation of 2 \texttimes{} n contingency tables}, Random Struct. Algorithms \textbf{13} (1998), no.~1, 71–79.

\bibitem{LawlerLimic}
Gregory~F. Lawler and Vlada Limic, \emph{Random walk: A modern introduction}, Cambridge University Press, 2010.

\bibitem{LevinPeres}
David~A. Levin and Yuval Peres, \emph{Markov chains and mixing times: Second edition}, American Mathematical Society, 2017.

\bibitem{LubetzkyPeres}
Eyal Lubetzky and Yuval Peres, \emph{Cutoff on all ramanujan graphs}, Geometric and Functional Analysis \textbf{26} (2015), 1190--1216.

\bibitem{LubetzkySly}
Eyal Lubetzky and Allan Sly, \emph{{Cutoff phenomena for random walks on random regular graphs}}, Duke Mathematical Journal \textbf{153} (2010), no.~3, 475 -- 510.

\bibitem{MatsuiMatsuiOno}
Tomomi Matsui, Yasuko Matsui, and Yoko Ono, \emph{Random generation of $2 \times 2 \times \cdots \times 2 \times j$ contingency tables}, Theoretical Computer Science \textbf{326} (2004), no.~1, 117--135.

\bibitem{MiloKashtanItzkovitzNewmanAlon}
Ron Milo, Nadav Kashtan, Shalev Itzkovitz, Mark E.~J. Newman, and Uri Alon, \emph{On the uniform generation of random graphs with prescribed degree sequences}, arXiv: Statistical Mechanics (2003).

\bibitem{Morris}
Ben Morris, \emph{Improved bounds for sampling contingency tables}, Randomization, Approximation, and Combinatorial Optimization. Algorithms and Techniques (Berlin, Heidelberg) (Dorit~S. Hochbaum, Klaus Jansen, Jos{\'e} D.~P. Rolim, and Alistair Sinclair, eds.), Springer Berlin Heidelberg, 1999, pp.~121--129.

\bibitem{Nestoridi}
Evita Nestoridi, \emph{A non-local random walk on the hypercube}, Applied Probability Trust \textbf{49} (2017), no.~4, 1288–1299.

\bibitem{NestoridiNguyen}
Evita Nestoridi and Oanh Nguyen, \emph{On the mixing time of the diaconis–gangolli random walk on contingency tables over $\mathbb{Z}/q\mathbb{Z}$}, Annales de l'Institut Henri Poincar{\'e}, Probabilit{\'e}s et Statistiques (2018).

\bibitem{RaoJanaBandyopadhyay}
A.~Ramachandra Rao, Rabindranath Jana, and Suraj Bandyopadhyay, \emph{A markov chain monte carlo method for generating random (0, 1)-matrices with given marginals}, Sankhyā: The Indian Journal of Statistics, Series A (1961-2002) \textbf{58} (1996), no.~2, 225--242.

\bibitem{TokhomirovYoussef}
Konstantin Tikhomirov and Pierre Youssef, \emph{Sharp poincaré and log-sobolev inequalities for the switch chain on regular bipartite graphs}, Probability Theory and Related Fields \textbf{185} (2022).

\bibitem{Wormald}
N.~C. Wormald, \emph{Models of random regular graphs}, London Mathematical Society Lecture Note Series (1999), 239–298.

\end{thebibliography}
\end{document}